\documentclass[12pt]{article}

\usepackage[english]{babel}
\usepackage{geometry}
\usepackage{pdfpages}
\usepackage{amsmath}
\usepackage{amssymb}
\usepackage{amsthm}
\usepackage{graphicx}
\usepackage{float}
\usepackage{array}
\usepackage{xcolor}
\usepackage{listings}
\usepackage{hyperref}

\definecolor{keywordcolor}{rgb}{0.7,0.1,0.1}
\definecolor{commentcolor}{rgb}{0.4,0.4,0.4}
\definecolor{symbolcolor}{rgb}{0.0,0.1,0.6}
\definecolor{sortcolor}{rgb}{0.1,0.5,0.1}

\usepackage{lstlean}

\lstdefinestyle{LeanCode}{
  language=lean,
  basicstyle=\ttfamily\fontsize{8pt}{8pt}\selectfont,
  keywordstyle=[1]{\color{keywordcolor}},
  keywordstyle=[2]{\color{sortcolor}},
  commentstyle=\itshape\color{commentcolor},
  stringstyle=\ttfamily,
  breaklines=true,
  breakatwhitespace=false,
  columns=fullflexible,
  keepspaces=true,
  showstringspaces=false,
  tabsize=2,
  extendedchars=false,
  captionpos=b
}

\usepackage{newunicodechar}
\newunicodechar{⦃}{\ensuremath{\{\!\{}}
\newunicodechar{⦄}{\ensuremath{\}\!\}}}
\newunicodechar{⋂}{\ensuremath{\bigcap}}
\newunicodechar{ₗ}{\ensuremath{_{l}}}
\newunicodechar{ᵢ}{\ensuremath{_{i}}}
\newunicodechar{ℝ}{\ensuremath{\mathbb{R}}}
\newunicodechar{ₜ}{\ensuremath{_{t}}}

\newtheorem{theorem}{Theorem}

\theoremstyle{definition}
\newtheorem{definition}{Definition}[section]

\theoremstyle{lemma}
\newtheorem{lemma}[theorem]{Lemma}

\theoremstyle{remark}

\theoremstyle{boldremark}

\title{Proof and More Variations of Bellman’s Lost-in-a-forest Problem}
\author{Zhipeng Deng}
\date{} 

\begin{document}
\maketitle

\begin{abstract}
In this paper, based on our previous general formulation and computational solution to Bellman’s Lost-in-a-forest Problem, we provide the proof of general solution and obtained more variations and results related to this problem. This paper provides generalized formalized method connecting curve covering, lost-in-the-forest problem, and traveling salesman problem with neighborhoods. We prove the equivalence and convergence. We also provide more results of searching for two lines, connection to Wetzel’s unit arc covering problem, variations with closed path, variations in three dimensions, etc. The results include general calculation equations, partial analytical results, and numerical results.
\end{abstract}

\\

\noindent \textbf{Classification}

Optimization and Control (math.OC)

Metric Geometry (math.MG)

Discrete Mathematics (cs.DM)

Computational Geometry (cs.CG)

49K30 (Optimal Solutions in Calculus of Variations)

49Q10 (Optimization of Shapes Other Than Minimal Surfaces)

52A40 (Inequalities and Extremum Problems)

\tableofcontents

\section{Introduction}

Bellman's lost-in-a-forest problem is a difficult unsolved minimization problem in geometry, introduced by Richard E. Bellman \cite{Gross1955} \cite{Croft2012} \cite{Finch2004}. It is commonly stated as follows:: 
\\

``A hiker is lost in a forest whose shape and dimensions are precisely known. What is the best path for the hiker to follow to escape the forest?"
\\

It is typically assumed that the hiker does not know either the starting location or the initial facing orientation. The optimal path is defined as the one that minimizes the worst-case distance that must be traveled before reaching the forest boundary. Proven solutions are known for only a limited number of forest shapes, including the straight line \cite{Gross1955} \cite{Gluss1961} \cite{Isbell1957} \cite{Joris1980} and the unit strip \cite{Zalgaller2005}.

In our previous paper \cite{Deng2024}, we proposed a general computational solution for Bellman’s Lost-in-a-forest Problem. The framework converted this problem to traveling salesman problem with neighborhoods (TSPN). The previous paper \cite{Deng2024} included some assumptions and results may not reach the optimal solution.

The contributions and primary content of this paper consist of proving the equivalence, deriving additional results using previously established general solution, comparing them with other studies, and extending other more variations.

\section{General formalized method connecting curve covering, Bellman's Lost-in-the-forest problem, and TSPN}

The literature \cite{Finch2004} has already established a connection between Bellman’s Lost-in-a-Forest Problem and the Moser’s Worm Problem: "the shortest escape path from a closed, convex region is the shortest path that does not fit in the interior of the region". These two problems exhibit a dual relationship. 

Traditionally, Bellman’s Lost-in-a-forest Problem is defined in minimax way \cite{Croft2012} \cite{Finch2004}\cite{Gluss1961}. Let the search space be $\mathbb{R}^2$. The forest boundary is $B$. The target boundary is characterized by parameters $\xi \in \Xi $. The goal is to find a path $C$ such that for every possible target $B(\xi)$, the path intersects the boundary. The formulation is

\begin{equation}
\mathop{\text{minimize}}\limits_{C} \mathop{\text{sup}}\limits_{\xi \in \Xi} L\left ( C, \xi  \right ) 
\end{equation}

where $L\left ( C, \xi  \right )$ is the length required to reach target boundary $B(\xi)$.

According to the general solution described in our paper \cite{Deng2024}, and based on Convex Analysis, specifically the properties of Convex Hulls and Geometric Duality. By realizing that "intersecting all forest boundaries" is geometrically identical to "drawing a curve whose convex hull contains the shape," the problem sheds its adversarial $\sup\inf$ nature. The worst-case scenario is structurally baked into the constraints of TSPN.

In our previous paper \cite{Deng2024}, we converted the problem to TSPN, and provided two weak forms of Bellman’s Lost-in-a-forest Problem. Weak form I is starting from a known point with unknown orientations. Weak form II is starting from several possible known points with unknown orientation. Discrete solution to Weak Form I can be written as

\begin{definition}[Permutation formulation]
Let $S_N$ denote the symmetric group on $\{0,1,\dots,N-1\}$.
For each permutation $\pi\in S_N$, define the path-length functional
\[
\mathcal L_N(\pi,p)
:=
\|p_{\pi(0)}-O\|
+
\sum_{i=1}^{N-1}\|p_{\pi(i)}-p_{\pi(i-1)}\|,
\]
where $O$ is (0,0).
\end{definition}

\begin{definition}[Weak Form I]
The discrete weak problem with one starting point and $N$ orientations is
\begin{equation}
\mathcal J_N^{(1)}
:=
\mathop{\text{minimize}}\limits_{\pi\in S_N, p_i\in \partial\Omega_i}\mathcal L_N(\pi,p)
\end{equation}
subject to the permutation constraints, where $\pi$ is induced by $S_N$.
\end{definition}

Similarly, discrete solution to Weak Form II can be written as
\begin{definition}[Two-parameter discrete family of transformed boundaries]
For each pair $(k,i)\in\{1,\dots,M\}\times\{0,\dots,N-1\}$, define
\[
\partial\Omega_{k,i} := \partial\Omega_{s_k,\theta_i},
\qquad
\theta_i:=\frac{2\pi i}{N}.
\]
Choose points
\[
p_{k,i}\in \partial\Omega_{k,i}.
\]
The total number of discrete escape points is $MN$.
\end{definition}

\begin{definition}[Weak Form II]
The formulation is
\begin{equation}
\mathcal J_{M,N}^{(2)}
:=
\mathop{\text{minimize}}\limits_{\pi\in S_{MN}, p_{k,i}\in\partial\Omega_{k,i}}
\|p_{\pi(1)}-O\|
+\sum_{h=2}^{MN}\|p_{\pi(h)}-p_{\pi(h-1)}\|
\end{equation}
\end{definition}

Based on these general computational solutions, this paper will provide proof and more results in following sections.

\section{Proof of general computational framework}
In this section, we provide the proof of general computational framework, including discretized Bellman's Lost-in-a-forest Problem is reducible to TSPN, and solution convergence of discretized to continuous problem. 

The proof framework consists of four steps: (1) Bellman’s Lost-in-a-forest Problem is equivalent to transformed boundary intersection (Corresponds to the proof of concept in our paper \cite{Deng2024}). (2) Existence of optimal path for Bellman’s Lost-in-a-forest Problem (Previous study \cite{Finch2004} has discussed existence in terms of the diameter of bounded forest boundary). (3) Discretized Bellman’s Lost-in-a-forest Problem is reducible to TSPN (Although NP-hard, global optimization algorithms exist, and MIP solvers can find and guarantee a global optimal). (4) Global optimal of discretized version TSPN converges to the global optimal of original continuous Bellman’s Lost-in-a-forest Problem (Similar to extreme value theorem, minimum can be found on the discrete grid, then as the grid is refined, the global optimum is obtained, thereby solving the original problem).

(1). We first provide a precise continuous formulation for Bellman's Lost-in-a-forest Problem as rigid-motion boundary intersection problem

Assumptions: Let the forest boundary be a nonempty compact closed set $B\subset \mathbb{R}^2$. Let the starting region be a nonempty compact set $R\subset \mathbb{R}^2$. Orientation set is $\Theta=[0,2\pi]$. Let $G:=R \times \Theta$ be the compact parameter space of unknown initial starting location and orientation.

Let $\mathcal{C}$ be the class of escape paths $f \colon [0,1] \to \mathbb{R}^2$ which is absolutely continuous in a canonical coordinate system with $f(0)=0$, and length $\mathcal{L}(f) := \int_0^1 \| f'(t) \| \, dt < \infty$. The actual physical path is obtained by rigid motion $p(t) = s + \mathrm{Rot}(\alpha) f(t)$, where $s \in R$ and $\alpha \in \Theta$

For rigid-motion transformation, for each $g = (s, \alpha) \in G$, define the rigid motion $T_g(x) := \operatorname{Rot}(-\alpha)(x - s)$ including translation and rotation, and define the transformed boundary $B_g := T_g(B)$. The rotation matrix is
\[
\operatorname{Rot}=
\begin{pmatrix}
\cos\alpha & -\sin\alpha\\
\sin\alpha & \cos\alpha
\end{pmatrix}
\]

Feasibility as boundary intersection constraint is to define the distance from a curve to a set:

\[
d(f, B_g) := \inf_{t \in [0,1]} \operatorname{dist}(f(t), B_g), \quad \operatorname{dist}(x, S) := \inf_{y \in S} \| x - y \|.
\]

\noindent Then ``$f$ intersects $B_g$'' is exactly $d(f, B_g) = 0$.

Thus, continuous Bellman's Lost-in-a-forest Problem can be written as
\[
 L^\star := \inf_{f \in \mathcal{C}} \{ \mathcal{L}(f) : d(f, B_g) = 0,\forall g \in G \} \tag{P}
\]

(2). Then we provide equivalence theorem - original Bellman's Lost-in-a-forest Problem is equivalent to transformed boundary intersection problem.

\begin{theorem}[Equivalence of formulation] Assume the hiker's search path is $f$ starting at origin with fixed orientation. Then
\[
p([0,1]) \cap B \neq \emptyset \iff f([0,1]) \cap B_{(s,\alpha)} \neq \emptyset.
\]

Consequently, the original minimax Bellman's Lost-in-a-forest Problem is equivalent to transformed boundary intersection problem.
\end{theorem}

\begin{proof}

$$p(t) \in B \iff s + \mathrm{Rot}(\alpha)f(t) \in B \iff f(t) \in \mathrm{Rot}(-\alpha)(B - s) = B_{(s, \alpha)}.$$
\end{proof}

Thus, intersection of $p$ with $B$ is equivalent to intersection of $f$ with the transformed boundary $B_g$ for the corresponding unknown $g = (s, \alpha)$. Under the rigid-motion transformation $B_g=T_g (B),g=(s,\alpha)\in G$, a canonical path $f$ guarantees escape if and only if 
$f([0,1]) \cap B_{(s,\alpha)} \neq \emptyset \forall  g \in G$. Therefore, Bellman's Lost-in-a-Forest Problem is equivalent to the constrained optimization problem.

(3). Now we provide existence of optimal solution to Bellman's Lost-in-a-forest Problem. Although in the literature \cite{Finch2004}, the authors proved the existence of optimal path for bounded forest boundary that ``If the forest is bounded, a line segment whose length is the diameter is surely an escape path''. In this paper, we did not have bounded forest boundary assumption.

\begin{theorem}[Existence of minimizer] Assume $B$ is compact nonempty and $R$ compact nonempty. If there exists at least one feasible curve $f \in \mathcal{C}$ for (P) with finite length, then (P) admits a minimizer $f^* \in \mathcal{C}$.
\end{theorem}

\begin{proof}

Let $(f_n) \subset \mathcal{C}$ be a minimizing sequence: $\mathcal{L}(f_n) \downarrow L^\star$. Then $\sup_n \mathcal{L}(f_n) < \infty$. Reparametrize each $f_n$ by length onto $[0, 1]$, so that each $f_n$ becomes Lipschitz with a uniform Lipschitz constant $K := \sup_n \mathcal{L}(f_n)$. Hence $\{f_n\}$ is equicontinuous and uniformly bounded on $[0, 1]$.

By Arzel\`a--Ascoli theorem, there exists a uniformly convergent subsequence $f_{n} \to f$ in $C([0,1]; \mathbb{R}^2)$.

For feasibility passes to the limit, for any fixed $g \in G$, define $\phi_g(t) := \operatorname{dist}(f(t), B_g)$, continuous in $t$ since $B_g$ is closed. Uniform convergence implies $\phi_g^{(n)}(t) := \operatorname{dist}(f_{n}(t), B_g)$ converges uniformly to $\phi_g(t)$. 

Since the distance function to a closed set is 1-Lipschitz,
$ |\phi_g^{({n})}(t) - \phi_g(t)| \le \|f_{{n}}(t) - f(t)\|. $ Uniform convergence $f_{{n}} \to f$ therefore implies uniform convergence $ \phi_g^{({{n}})} \to \phi_g. $ Consequently,
$$ d(f_{n}, B_g) = \inf_{t \in [0,1]} \phi_g^{({n})}(t) \longrightarrow \inf_{t \in [0,1]} \phi_g(t) = d(f, B_g). $$

Since $d(f_{n}, B_g) = 0$ for every $n$, it follows that
$ d(f, B_g) = 0. $ This holds for every $g \in G$. Thus $f$ is feasible.

As for lower semicontinuity of length, since the sequence $\{f_n\}$ is uniformly Lipschitz, it is bounded in $W^{1,\infty}(0, 1; \mathbb{R}^2)$. By compactness, there exists a subsequence such that
$ f_n \rightharpoonup f \text{in } W^{1,1}(0, 1; \mathbb{R}^2). $ The arc-length functional $ \mathcal{L}(f) = \int_0^1 |f'(t)| dt $ is lower semicontinuous under weak convergence, hence
$$ \mathcal{L}(f) \le \liminf_{n\to\infty} \mathcal{L}(f_n)= L^\star. $$

Therefore  $\mathcal{L}(f) \le L^\star. $ Since $f$ is feasible, it follows that
 $\mathcal{L}(f) = L^\star $ and $f$ is a minimizer.
\end{proof}

(4). We now prove that finite discretization yields TSPN and state precisely what a finite global certificate proves.

Throughout (4)--(7), retain the nonempty compact sets $B$ and $R$ and the state space $G=R\times[0,2\pi]$ from (1), and assume, as in Theorem~2, that the continuum-feasible set
\[
 \mathcal F:=\{f\in\mathcal C:d(f,B_g)=0\text{ for every }g\in G\}
\]
is nonempty. Thus $0\le L^\star<\infty$; compactness of $B$ in particular includes boundedness. To complete the minimax interpretation in (2), parametrize a strategy by arclength and let $T$ be the finite supremum of its first boundary-hitting arclengths. Closedness of each $B_g$ ensures that each first hit is attained. Truncation at $T$ still meets every $B_g$ and has length $T$ (use the constant path if $T=0$). Conversely, every hitting length is at most the length of a feasible path. Taking infima in both directions therefore identifies the minimax value with $L^\star$.

Choose nonempty finite sets satisfying
\[
 G_m\subseteq G_{m+1}\subseteq G,
 \qquad \overline{\bigcup_{m\ge1}G_m}=G.
\]
Define
\[
 \mathcal F_m:=\{f\in\mathcal C:d(f,B_g)=0\text{ for every }g\in G_m\},
 \qquad L_m^\star:=\inf_{f\in\mathcal F_m}\mathcal L(f).
 \tag{$\mathrm P_m$}
\]

\begin{lemma}[Monotonicity and lower bounds]
\label{lem:finite-monotonicity}
For every $m$, $0\le L_m^\star\le L_{m+1}^\star\le L^\star$.
\end{lemma}
\begin{proof}
The inclusions $\mathcal F\subseteq\mathcal F_{m+1}\subseteq\mathcal F_m$ give the stated inequalities by taking infima of the nonnegative length functional.
\end{proof}

Write $G_m=\{g_1,\ldots,g_{H_m}\}$ and $S_h:=B_{g_h}$. For a permutation $\pi$ of $\{1,\ldots,H_m\}$ and points $p_h\in S_h$, put
\[
 \ell(\pi,p):=\|p_{\pi(1)}\|
       +\sum_{k=2}^{H_m}\|p_{\pi(k)}-p_{\pi(k-1)}\|.
\]

\begin{lemma}[Polyline optimality for prescribed visits]
\label{lem:polyline-optimality}
Among rectifiable curves starting at $0$ and visiting $p_{\pi(1)},\ldots,p_{\pi(H_m)}$ in this order, the minimum length is $\ell(\pi,p)$, attained by the corresponding polyline.
\end{lemma}
\begin{proof}
Choose nondecreasing visit times, allowing equal times for coincident visits. The length of each intervening subarc is at least the distance between its endpoints. Summing, including the subarc from $0$ to the first visit, gives the lower bound $\ell(\pi,p)$. The polyline attains this bound, including when some segments have zero length.
\end{proof}

\begin{theorem}[Exact finite TSPN representation]
\label{thm:finite-tspn}
For every $m$, the finite optimum is attained and equals
\[
 L_m^\star=\min_{\pi\in S_{H_m},\ p_h\in S_h}\ell(\pi,p).
 \tag{$\mathrm{TSPN}_m$}
\]
\end{theorem}
\begin{proof}
Every selected polyline belongs to $\mathcal F_m$, so $L_m^\star$ is at most the TSPN infimum. Conversely, for $f\in\mathcal F_m$, compactness of $f([0,1])$ and closedness of $S_h$ give visit times $t_h$ with $f(t_h)\in S_h$. Order these times nondecreasingly, breaking ties arbitrarily, and set $p_h=f(t_h)$. Lemma~\ref{lem:polyline-optimality} gives $\ell(\pi,p)\le\mathcal L(f)$. Taking infima proves equality. Each $S_h$ is nonempty compact; for each permutation, $\ell$ is continuous on the nonempty compact product $\prod_h S_h$. There are finitely many permutations, so the minimum is attained. Its polyline is a finite global minimizer.
\end{proof}

The TSPN above is an exact finite-dimensional geometric problem. An exact MIP, MIQCP, or other global-optimization representation must be justified for the particular neighborhoods; compactness alone does not supply such a representation.

\begin{theorem}[Certified bounds for a finite discretization]
\label{thm:finite-certificate}
Suppose the finite model represents $(\mathrm{TSPN}_m)$ exactly. Let $D_m$ be a verified global lower bound, and let $f_m\in\mathcal F_m$ be a verified incumbent with a certified length bound $\mathcal L(f_m)\le I_m$. Then
\[
 D_m\le L_m^\star\le\mathcal L(f_m)\le I_m.
\]
If $D_m=I_m$, then $f_m$ is a global minimizer of the finite problem and all four quantities agree.
\end{theorem}
\begin{proof}
Validity of the lower bound and feasibility of the incumbent give the first two inequalities; the last is assumed. Equality of the endpoints forces equality throughout. Thus a numerical solver report is an exact certificate only after model validity, feasibility, and the bounds, including numerical errors, have been verified. Finite-state feasibility alone does not make $I_m$ an upper bound for $L^\star$.
\end{proof}

(5). We now prove convergence of the exact finite optima and recovery of continuous global minimizers.

\begin{theorem}[Value convergence and recovery of global minimizing paths]
\label{thm:convergence-recovery}
For the nested dense hierarchy above,
\[
 L_m^\star\uparrow L^\star,
 \qquad L^\star=\sup_{m\ge1}\operatorname{OPT}(\mathrm{TSPN}_m).
\]
If $f_m$ is any global minimizer of the $m$th finite problem, reparametrized by constant speed on $[0,1]$, then $(f_m)$ is precompact in the uniform topology. Every uniformly convergent subsequence $f_{m_j}\to f^\star$ with $m_j\to\infty$ satisfies
\[
 f^\star\in\mathcal F,
 \qquad \mathcal L(f^\star)=L^\star.
\]
\end{theorem}
\begin{proof}
Lemma~\ref{lem:finite-monotonicity} gives a limit $L_\infty:=\lim_m L_m^\star\le L^\star$. Finite minimizers exist by Theorem~\ref{thm:finite-tspn}. Constant-speed reparametrization preserves the image, initial point, and length; a zero-length curve is represented by the constant path. Hence, for all $s,t\in[0,1]$,
\[
 \|f_m(t)-f_m(s)\|\le L_m^\star|t-s|\le L^\star|t-s|,
 \qquad \|f_m(t)\|\le L^\star.
\]
These estimates include $L^\star=0$. Arzel\`a--Ascoli gives uniform precompactness. Every uniform limit $f$ along increasing indices is $L^\star$-Lipschitz with $f(0)=0$, so $f\in\mathcal C$.

Fix such a subsequence $f_{m_j}\to f$. For any nonempty closed set $S$, the distance function is $1$-Lipschitz, and consequently
\[
 |d(f_{m_j},S)-d(f,S)|\le\|f_{m_j}-f\|_\infty.
\]
For $g\in\bigcup_m G_m$, nestedness gives $g\in G_{m_j}$ eventually. Therefore $d(f,B_g)=0$ on this dense union.

For arbitrary $g=(s,\alpha)\in G$, choose $g_n=(s_n,\alpha_n)\in\bigcup_mG_m$ with $g_n\to g$, and choose $x_n\in f([0,1])\cap B_{g_n}$. Compactness of the curve image gives a subsequence $x_{n_k}\to x\in f([0,1])$. By the definition of transformed boundaries,
\[
 y_{n_k}:=s_{n_k}+\operatorname{Rot}(\alpha_{n_k})x_{n_k}\in B,
 \qquad y_{n_k}\longrightarrow s+\operatorname{Rot}(\alpha)x.
\]
Closedness of $B$ implies $s+\operatorname{Rot}(\alpha)x\in B$, hence $x\in B_g$. Thus $f\in\mathcal F$.

For completeness, lower semicontinuity of length follows directly from partitions. For every $0=t_0<\cdots<t_q=1$, uniform convergence gives
\[
 \sum_{i=1}^q\|f(t_i)-f(t_{i-1})\|
 =\lim_{j\to\infty}\sum_{i=1}^q\|f_{m_j}(t_i)-f_{m_j}(t_{i-1})\|
 \le\liminf_{j\to\infty}\mathcal L(f_{m_j})=L_\infty.
\]
Taking the supremum over partitions, which equals the integral length for absolutely continuous curves, yields $\mathcal L(f)\le L_\infty$. Feasibility now gives
\[
 L^\star\le\mathcal L(f)\le L_\infty\le L^\star.
\]
All quantities agree. This proves value convergence and, since the convergent subsequence was arbitrary, the asserted recovery of every uniform cluster point. No uniqueness or convergence of the entire sequence is required.
\end{proof}

(6). We now combine finite certificates with certified continuum-feasible upper bounds.

\begin{theorem}[Two-sided certificates and the certified computational hierarchy]
\label{thm:continuum-certificate}
Let $\widehat f_m\in\mathcal F$ be paths whose feasibility is certified for every $g\in G$, and let $U_m\ge\mathcal L(\widehat f_m)$ be finite certified length bounds. Define
\[
 \underline L_m:=L_m^\star,
 \qquad \overline L_m:=\min_{1\le j\le m}U_j.
\]
Then $\underline L_m\le L^\star\le\overline L_m$, $\underline L_m\uparrow L^\star$, and $(\overline L_m)$ is nonincreasing. If $\overline L_m-\underline L_m\le\varepsilon$, a candidate attaining the least bound $\overline L_m$ is $\varepsilon$-optimal. Equality of the bounds proves exact global optimality. If additionally $\inf_m U_m=L^\star$, then
\[
 \overline L_m\downarrow L^\star,
 \qquad \overline L_m-\underline L_m\longrightarrow0,
\]
so every prescribed $\varepsilon>0$ is met at a finite stage.

More generally, verified finite lower bounds $D_m$ can be used by setting $\underline L_m:=\max_{1\le j\le m}D_j$. The same interval and gap certificates hold; the lower bounds still converge to $L^\star$ if the certified finite gaps $I_m-D_m$ tend to zero as in Theorem~\ref{thm:finite-certificate}. In particular, any finite global minimizer that is also continuum feasible is already a continuous global minimizer.
\end{theorem}
\begin{proof}
Each $L_m^\star$ is a lower bound by Lemma~\ref{lem:finite-monotonicity}, whereas continuum feasibility gives $L^\star\le\mathcal L(\widehat f_j)\le U_j$ for every $j$. Choose $j_m\le m$ attaining the finite minimum of the $U_j$. Then
\[
 0\le\mathcal L(\widehat f_{j_m})-L^\star
 \le\overline L_m-L^\star
 \le\overline L_m-\underline L_m.
\]
This proves both the positive-gap and zero-gap assertions. Theorem~\ref{thm:convergence-recovery} gives convergence of the exact lower bounds. If $\inf_mU_m=L^\star$, then for each $\delta>0$ some $U_j<L^\star+\delta$; for all $m\ge j$,
$L^\star\le\overline L_m\le U_j<L^\star+\delta$.
This proves convergence of the upper bounds, hence of the gap and the finite-stage assertion for every positive tolerance.

For nonzero finite gaps $\delta_m:=I_m-D_m$, Theorem~\ref{thm:finite-certificate} gives
\[
 L_m^\star-\delta_m\le D_m\le L_m^\star\le L^\star.
\]
If $\delta_m\to0$, then $D_m\to L^\star$; its running maxima are nondecreasing, bounded above by $L^\star$, and at least $D_m$, so they also tend to $L^\star$. The same interval proof applies. Finally, if a finite optimizer $f_m$ lies in $\mathcal F$, then $L_m^\star\le L^\star\le\mathcal L(f_m)=L_m^\star$, proving exact identification.
\end{proof}

Thus the hierarchy consists of refining the state sets, certifying finite lower bounds, independently certifying continuum-feasible candidates, and retaining the best bounds. Finite sampling supplies no continuum upper certificate by itself. The finite-stage conclusion is conditional on $\inf_mU_m=L^\star$; an executable algorithm also requires effective geometric representations and terminating certification procedures. Exact equality need not occur at any finite stage. The next part states explicit sufficient effective data for a positive-tolerance algorithm.

(7). Quantitative convergence and a certified $\varepsilon$-algorithm.

Choose $\rho\ge\max\{1,\sup_{b\in B,s\in R}\|b-s\|\}$ and, on the original state space, set
\[
 \begin{aligned}
 d_G((s,\alpha),(s',\alpha'))&:=\|s-s'\|+\rho|\alpha-\alpha'|,\\
 h_m&:=\sup_{g\in G}\min_{g'\in G_m}d_G(g,g'),
 \qquad \mathcal E(f):=\sup_{g\in G}d(f,B_g).
 \end{aligned}
\]
Compactness and nested density give $h_m\downarrow0$: any finite cover by balls of radius $\delta>0$ centred in $\bigcup_mG_m$ has all its centres in some $G_M$. Moreover, pairing the same $b\in B$ in two transformed boundaries and using
$\|\operatorname{Rot}(\alpha)-\operatorname{Rot}(\alpha')\|_2\le|\alpha-\alpha'|$
gives $d_H(B_g,B_{g'})\le d_G(g,g')$, where $d_H$ is Hausdorff distance. A visit to a nearest sampled boundary therefore gives $d(f,B_g)\le h_m$ for $f\in\mathcal F_m$, so $\mathcal E(f)\le h_m$.

Suppose supplied repair data consist of $\eta_0>0$, a nondecreasing $\omega:[0,\eta_0]\to[0,\infty)$ with $\omega(0)=0$ and $\omega(\eta)\to0$ as $\eta\downarrow0$, and an operator $\mathcal R$ such that
\[
 f\in\mathcal C,\quad\mathcal E(f)\le\eta\le\eta_0
 \quad\Longrightarrow\quad
 \mathcal R(f,\eta)\in\mathcal F,\qquad
 \mathcal L(\mathcal R(f,\eta))\le\mathcal L(f)+\omega(\eta).
\]

\begin{theorem}[Quantitative bound and effective positive-tolerance algorithm]
\label{thm:quantitative-certificate}
Whenever $h_m\le\eta_0$,
\[
 0\le L^\star-L_m^\star\le\omega(h_m).
\]
For a finite-state feasible incumbent $f_m\in\mathcal F_m$ with certificate $D_m\le L_m^\star\le\mathcal L(f_m)\le I_m$, the repaired path $\widehat f_m:=\mathcal R(f_m,h_m)$ satisfies
\[
 D_m\le L^\star\le\mathcal L(\widehat f_m)\le I_m+\omega(h_m),
\]
\[
 0\le\mathcal L(\widehat f_m)-L^\star
 \le(I_m-D_m)+\omega(h_m).
\]
In particular, $\omega(\eta)\le C\eta^\beta$ gives the rate $L^\star-L_m^\star\le Ch_m^\beta$ for $C,\beta>0$.

For an effective algorithm, suppose the grids and repair are produced by terminating procedures in specified finite representations, $\eta_0$ is positive rational, and finite procedures supply rational bounds $a_m\ge h_m$, $a_m\to0$, and $r_m\ge\omega(a_m)$ whenever $a_m\le\eta_0$, with $r_m\to0$. Suppose also that, for each rational $\delta>0$, a finite global solver terminates with rational certified $D_m,I_m$ and an incumbent as above satisfying $I_m-D_m\le\delta$. Then every rational $\varepsilon>0$ admits a terminating certified algorithm returning a continuum-feasible $\widehat f$ and rational $D,U$ with
\[
 D\le L^\star\le\mathcal L(\widehat f)\le U,
 \qquad U-D\le\varepsilon.
\]
\end{theorem}
\begin{proof}
Repairing an exact finite minimizer gives
$L_m^\star\le L^\star\le L_m^\star+\omega(h_m)$.
Repairing the incumbent gives the second chain; subtracting $L^\star\ge D_m$ gives its error bound. For the algorithm, search for $m$ with $a_m\le\eta_0$ and $r_m\le\varepsilon/2$, solve with $\delta=\varepsilon/2$, and set
\[
 \widehat f:=\mathcal R(f_m,a_m),\qquad D:=D_m,
 \qquad U:=I_m+r_m.
\]
The search terminates since $a_m,r_m\to0$, and its tests compare rational numbers. The residual bound and repair guarantee give continuum feasibility and the stated interval, with
$U-D\le\varepsilon/2+\varepsilon/2=\varepsilon$.
All remaining steps terminate by the effective hypotheses. A specified repair modulus and effective repair procedure are additional input; qualitative compactness alone supplies neither a specified computable rate nor this algorithm.
\end{proof}
This is the global optimality certification. Therefore, under the stated modeling assumptions and discretization convergence framework, the problem admits a certified approximation scheme whose discretized solutions converge to the original continuous optimum of Bellman’s Lost-in-a-forest Problem.

\section{Results of searching for two lines}
In this section, we consider the case of searching for two lines with angle $\beta$ from the angle bisector in the middle. 

Original general solution equation of searching for two lines from the angle bisector in the middle with $\frac{1}{2}$ distance from the previous paper \cite{Deng2024} is
\begin{equation}
\begin{split}
\text{minimize} &\left\| p_{a_0}-O\right\| +\sum_{i=1}^{N-1}\left\| p_{a_i}-p_{a_{i-1}}\right\| \\
\text{subject to:} &p_{a_i}=(x_{a_i},y_{a_i}), O=(0,0),\\
&\left[-x_{a_i} \sin \frac{2\pi i}{N} + y_{a_i} \cos \frac{2\pi i}{N} + \frac{1}{2}\right]\cdot\left[-x_{a_i} \sin \left( \frac{2\pi i}{N} + \beta \right) + y_{a_i} \cos \left( \frac{2\pi i}{N} + \beta \right) + \frac{1}{2}\right] = 0, \\
&\forall i \in \{0, 1, 2, \ldots, N-1\}.\\
&\{a_0, a_1, \ldots, a_{N-1}\} \in S_N
\end{split}
\end{equation}

where $S_N$ is the symmetric group of all permutations of $\{0, 1, 2, \ldots, N-1\}$.

By solving above optimization, Table 1 lists some numerical results for various $\beta$ and various optimal path shapes.

\begin{table}[H]
    \centering
    \begin{tabular}{|c|c|c|c|c|}
         \hline
         Angle&  Arc+Tangent&  2 line segments&  3 line segments& 4 line segments\\\hline
         $0\pi/180$ &  1.62782&  1.66509&  1.63224& 1.63224\\\hline
         $10\pi/180$ &  1.71509&  1.78534&  1.72488& 1.72349\\\hline
         $20\pi/180$ &  1.80236&  1.91912&  1.82044& 1.80844\\\hline
         $30\pi/180$ &  1.88962&  2.06918&  1.91943& 1.91919\\\hline
         $40\pi/180$ &  1.97689 &  2.23903 &  2.02237 &  1.99290\\\hline
         $50\pi/180$ &  2.06416 &  2.43327 &  2.12981 &  2.08786\\\hline
         $60\pi/180$ &  2.15142 &  2.65802 &  2.24235 &  2.18428\\\hline
         $70\pi/180$ &  2.23869&  2.51003&  2.27414 &  2.26938\\\hline
         $75\pi/180$ &  2.28232&  2.40805&  2.25994 &  2.25994\\\hline
         $80\pi/180$ &  2.32596&  2.30995&  2.23975 &  2.23976\\\hline
         $85\pi/180$ &  2.36959&  2.21262&  2.19793 &  2.19878\\\hline
         $90\pi/180$ &  2.41322&  2.12132&  2.12132 &  2.12132\\\hline
         $95\pi/180$ &  2.45686&  2.17768 &  2.17768 &  2.17768\\\hline
         $100\pi/180$ &  2.50049&  2.28429 &  2.28428 &  2.28429\\\hline
         $105\pi/180$ &  2.54412&  2.40964 &  2.38526 &  2.38424\\\hline
         $110\pi/180$ &  2.58776&  2.55515 &  2.46098 &  2.45504\\\hline
         $115\pi/180$ &  2.63139&  2.72541 &  2.54224 &  2.54224\\\hline
         $120\pi/180$ &  2.675&  2.92649 &  2.57673 &  2.57673\\\hline
         $125\pi/180$ &  2.71866&  3.16673 &  2.59000 &  2.59000\\\hline
         $130\pi/180$ &  2.76229&  3.45778 &  2.65230 &  2.65233\\\hline
         $135\pi/180$ &  2.80592&  3.81647 &  2.73816 &  2.73116\\\hline
         $140\pi/180$ &  2.84956&  4.26809 &  2.83567 &  2.80107\\\hline
         $150\pi/180$ &  2.93682&  5.63557 &  3.07218 &  2.91254\\\hline
         $160\pi/180$ &  3.02409&  - &  3.38014 &  3.08883\\\hline
         $180\pi/180$ &  3.19862 &  - &  4.34784 &  3.59635\\\hline
    \end{tabular}
    \caption{Numerical results of searching for two lines with angle $\beta$. The path can be arc and tangent line, or line segments.}
    \label{tab:placeholder}
\end{table}

We plot results in Figure 1, including arc and tangent line, and 3-6 line segments. Shortest path should be the minimum value for each angle.

It is noted that similar approaches allow for the construction of more line segments, and finer numerical calculations and segmentation can yield more accurate results. 

\newpage
\begin{figure}[H]
    \centering
    \includegraphics[width=1\linewidth]{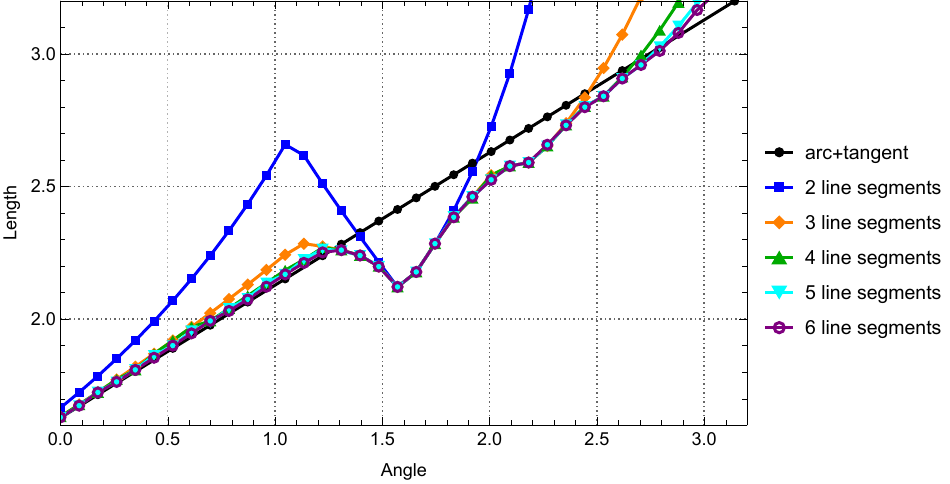}
    \caption{Length of optimal path in various angles $\beta$ of searching two lines}
\end{figure}

For arc and tangent line, similar to Section 4.2.9 of the previous paper \cite{Deng2024}, the continuous-form equations can be obtained as

\begin{equation}
\begin{aligned}
\text{minimize} &\sqrt{x(0)^2+y(0)^2}+\int_{0}^{\beta +\pi}  \sqrt{x'(t)^2+y'(t)^2}dt\\
\text{subject to:} & x(t) \cos t + y(t) \sin t - 1 = 0,\forall t\in \left [ 0, \beta +\pi\right ]
\end{aligned}
\end{equation}

Euler-Lagrange equations can be used to find the minimum of above functional. It serves as an upper bound for the solution in Figure 1.

The analytical solution for arc and tangent line in second column of Table 1 can be written as
\begin{equation}
1/2+\sqrt{3}/2+\pi/12+\beta/2
\end{equation}

This result can also be derived from the well-known search for a single line with unit distance.

For multi-segment polyline, we have the following conversion of coverage solution equation below:

For line segments, write the points in polar coordinates
\begin{equation}
P_i=(R_i\cos\phi_i,R_i\sin\phi_i), i=1, 2, 3, ...
\end{equation}

Let $R_1 = \frac{1}{2\cos\gamma_1}$, $R_2 = \frac{1}{2\cos\gamma_2}$, polyline length $\left | OP_1 \right | +\left | P_1P_2 \right | $ can be written as 
\begin{equation}
L_{2lines} = R_1 + \sqrt{R_1^2 + R_2^2 - 2R_1R_2 \cos(\phi_2 - \phi_1)} 
\end{equation}

Then we can solve two-line optimal path by
\begin{equation}
\begin{array}{l}
\mathop{\text{minimize}}\limits_{\gamma_i, \phi_i } L_{2lines} \\
\text{subject to:}0\le \gamma_i  \le \pi/2, 0 \le \phi_i \le 2 \pi,\\
([\phi_1 - \gamma_1, \phi_1 + \gamma_1] \\
\cup [\phi_1 - \gamma_1 + \beta, \phi_1 + \gamma_1 + \beta] \\
\cup [\phi_2 - \gamma_2, \phi_2 + \gamma_2] \\
\cup [\phi_2 - \gamma_2 + \beta, \phi_2 + \gamma_2 + \beta]) \bmod 2\pi = [0, 2\pi]
\end{array}
\end{equation}
The constraint is that the rotation angle covers the entire range $[0,2\pi]$.
And without losing generality, we can set $\phi_1=0$.

The corresponding results are shown by the blue line in Figure 1, and in the third column of Table 1.

For three line segments, let $R_1 = \frac{1}{2\cos\gamma_1}$, $R_2 = \frac{1}{2\cos\gamma_2}$, $R_3 = \frac{1}{2\cos\gamma_3}$, polyline length $\left | OP_1 \right | +\left | P_1P_2 \right |+\left | P_2P_3 \right | $ can be written as 
\begin{equation}
L_{3lines} = R_1 + \sqrt{R_1^2 + R_2^2 - 2R_1R_2\cos(\phi_2 - \phi_1)} + \sqrt{R_2^2 + R_3^2 - 2R_2R_3\cos(\phi_3 - \phi_2)}
\end{equation}
Then we can find three-line optimal path by
\begin{equation}
\begin{array}{l}
\mathop{\text{minimize}}\limits_{\gamma_i, \phi_i } L_{3lines}\\
\text{subject to:}0\le \gamma_i  \le \pi/2, 0 \le \phi_i \le 2 \pi,\\
\begin{aligned}
( & [\phi_1 - \gamma_1, \phi_1 + \gamma_1] \\
& \cup [\phi_1 - \gamma_1 + \beta, \phi_1 + \gamma_1 + \beta] \\
& \cup [\phi_2 - \gamma_2, \phi_2 + \gamma_2] \\
& \cup [\phi_2 - \gamma_2 + \beta, \phi_2 + \gamma_2 + \beta] \\
& \cup [\phi_3 - \gamma_3, \phi_3 + \gamma_3] \\
& \cup [\phi_3 - \gamma_3 + \beta, \phi_3 + \gamma_3 + \beta] ) \bmod 2\pi = [0, 2\pi]
\end{aligned}
\end{array}
\end{equation}
And without losing generality, we can set $\phi_1=0$.

Four line segments and more are similar. Additional detailed results are shown in Figure 2. 

\begin{figure}[H]
  \centering
  (a)\includegraphics[width=4cm]{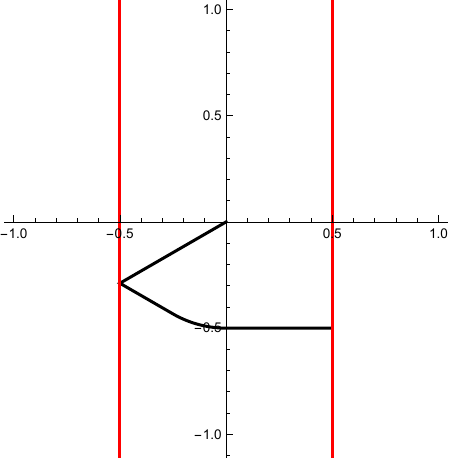}
  (b)\includegraphics[width=4cm]{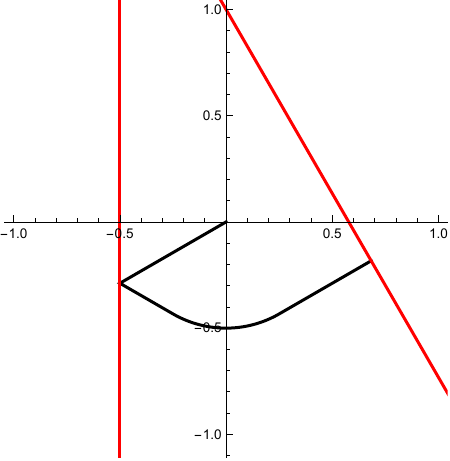}
  (c)\includegraphics[width=4cm]{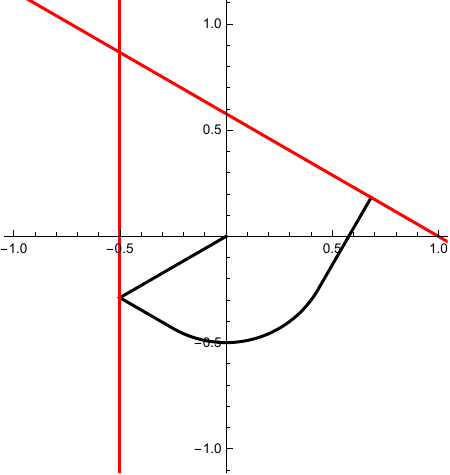}
  (d)\includegraphics[width=4cm]{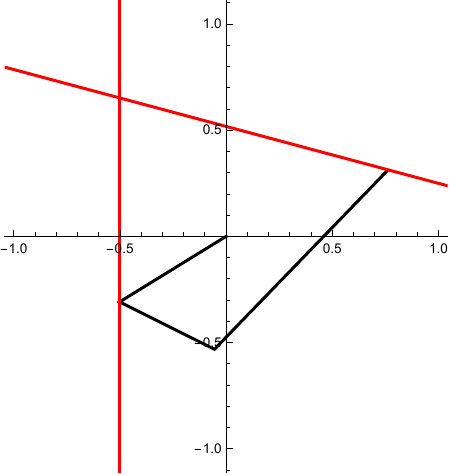}
  (e)\includegraphics[width=4cm]{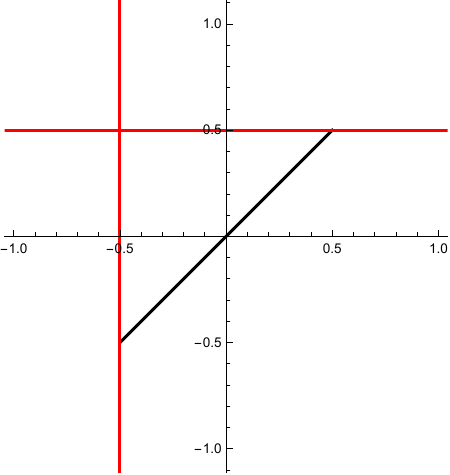}
  (f)\includegraphics[width=4cm]{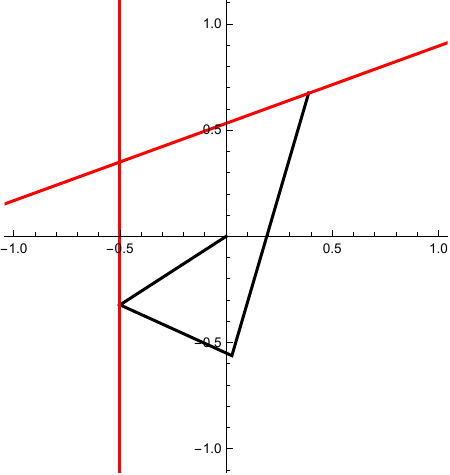}
  (g)\includegraphics[width=4cm]{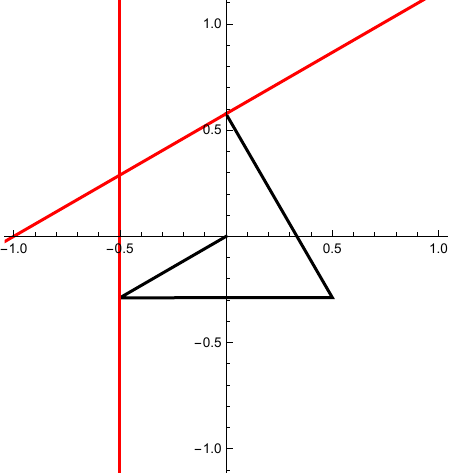}
  (h)\includegraphics[width=4cm]{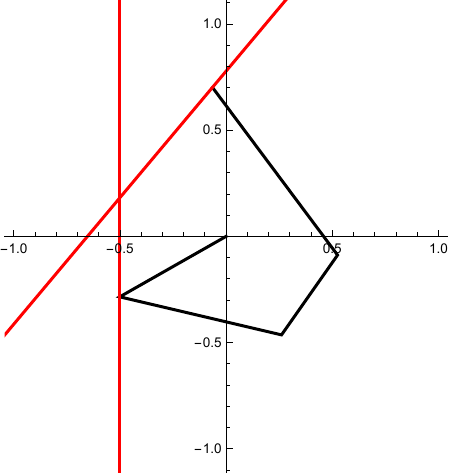}
  (i)\includegraphics[width=4cm]{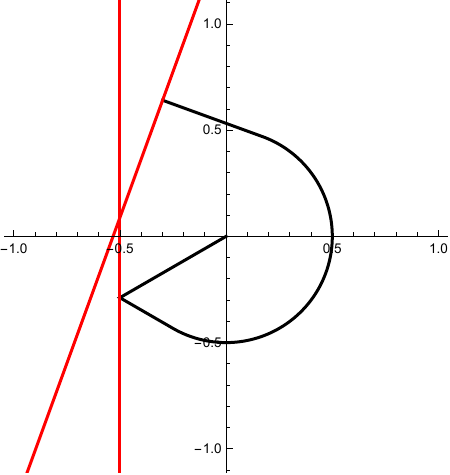}
  (j)\includegraphics[width=4cm]{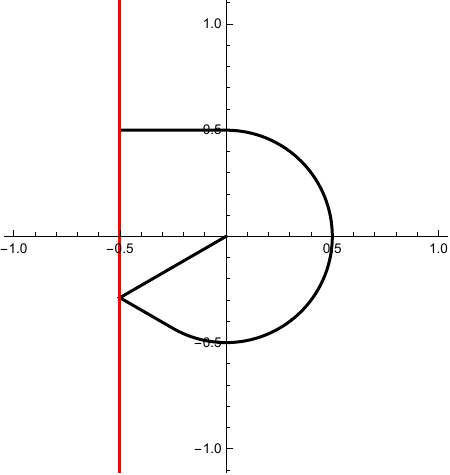}
  \caption{More examples of detailed results of searching for two lines with from the angle bisector in the middle at certain angle $0\pi/180$, $30\pi/180$, $60\pi/180$, $75\pi/180$, $90\pi/180$, $110\pi/180$, $120\pi/180$, $140\pi/180$, $160\pi/180$, $180\pi/180$. Among them, $90\pi/180$ is two line segments, $75\pi/180$,$110\pi/180$, $120\pi/180$ are three line segments, $140\pi/180$ is four line segments. (Black curve is escape path, and red curve is forest boundary)}
\end{figure}

The results for the two-line search problem developed in this section provide useful insights into Bellman’s Lost-in-a-Forest Problem for triangles/three lines, the fitting of a curve or polyline within a triangle, and the covering of a curve or polyline by a triangle. The rich and complex structures that already emerge in the two-line case further foreshadow the more complicated scenarios that arise for triangles and motivate their systematic analysis and classification, including important special cases such as isosceles triangles \cite{Gibbs2016}. While there are some results for regular polygons, relatively little is known about searching or escaping arbitrary polygons \cite{Finch2004} \cite{Gibbs2016-1} \cite{Kübel2021} \cite{Wetzel2003}.

\section{Connection to Wetzel’s problem/unit arc covering problem}

\begin{figure}[H]
    \centering
    \includegraphics[width=1\linewidth]{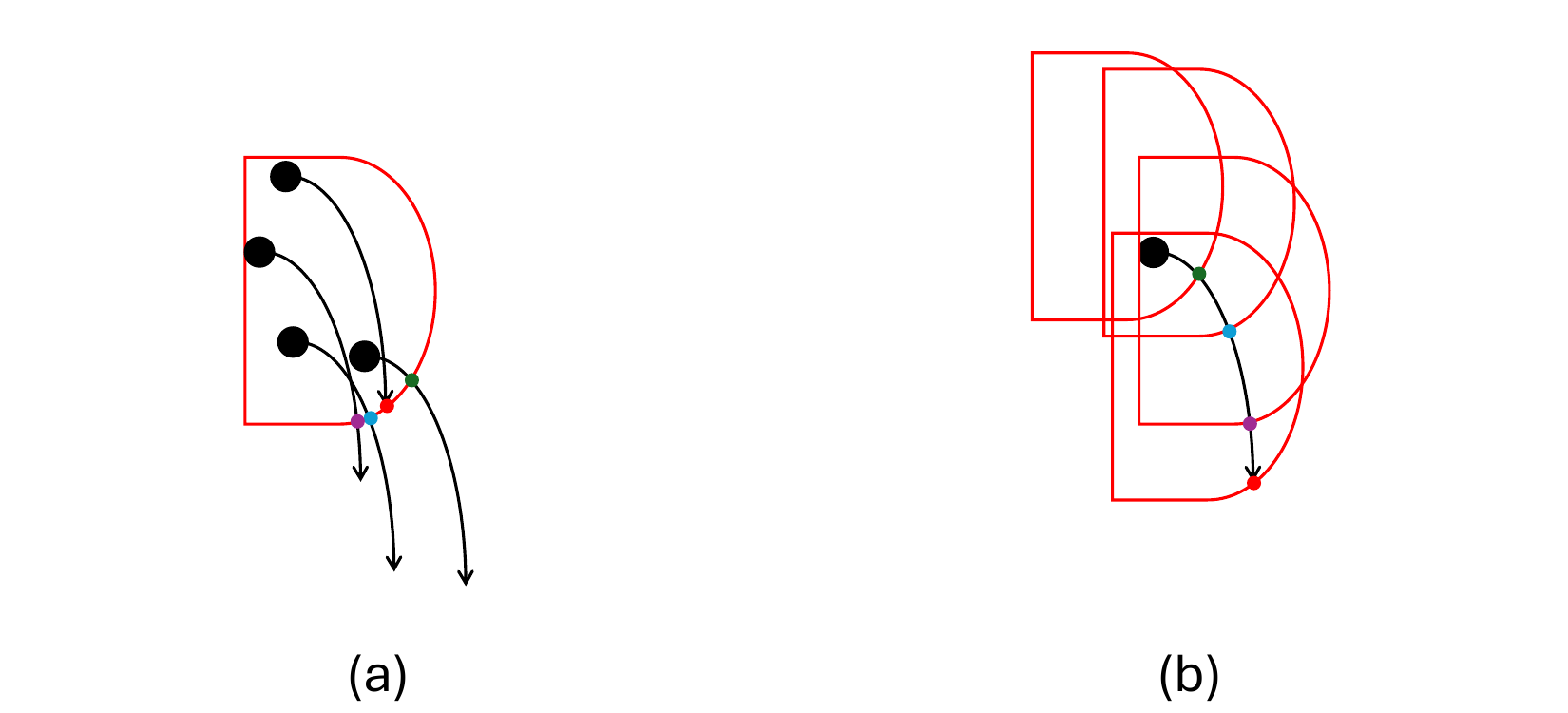}
    \caption{Weak Form IB - Starting from unknown points. We convert the path out of the forest boundary in different locations to the translation of the forest boundary. The translated escape points are all on the path. (Black curve represents the escape path, black dot is the starting point, colorful dots are the escape points, and red curve is the forest boundary)}
\end{figure}

Similar to Moser’s worm problem, Wetzel’s problem/unit arc covering problem is to find the smallest region that can contain a congruent copy of every planar arc of length 1 \cite{Wetzel1973}  \cite{Brass2005} . Different from the worm problem, the curve here can only move in translation and cannot rotate.

This corresponds to another variant/weak form of the Bellman’s Lost-in-a-forest Problem: the hiker does not know the location, but knows the direction, what is the shortest path to search?

In the previous paper \cite{Deng2024}, we discussed two weak forms of the Bellman’s Lost-in-a-forest Problem. Here we give the third weak form, which corresponds to Wetzel’s unit arc covering problem, and general solution can be obtained in a similar way.

\begin{definition}[Weak Form IB - Starting from unknown points with known orientation]
\end{definition}

Recall Figure 1 in our previous paper \cite{Deng2024}, the solution involves determining the shortest path passing through the escape points starting from the origin (0, 0). But the location and order of the points is unknown and needs to be solved.

Figure 3 demonstrate the proof of concept for Wetzel’s unit arc covering problem. We can transform the escape path from the forest boundary across various locations into a path that does not move by the translation of the forest boundary. With this perspective, we establish that for any translation, the escape path must pass through an escape point. Consequently, we reformulate the original problem into finding the shortest path that connects all possible escape points. This transformation allows us to formulate this problem as a TSPN.

\section{Extension to Lost-in-a-Forest Problem with closed path, and Moser’s worm problem for closed curve}
As for Bellman’s Lost-in-a-Forest Problem with closed curves, only a few papers discuss them, but most remain open problems \cite{Brass2005}  \cite{Ghomi2017} \cite{Melzak2007} \cite{Chakerian1973} \cite{Füredi2011} \cite{Movshovich2025}. 

In our previous paper \cite{Deng2024}, we already provided the general solution equation for a closed curve. The formulation is obtained by adding one additional term to the objective function. Recalling Definitions 2.2 and 2.4 from \cite{Deng2024}, the general equations for Weak Forms I and II are as follows:

\[
\mathop{\text{minimize}}\limits_{\pi\in S_N,p_i\in \partial\Omega_i} \mathcal \|p_{\pi(0)}-O\|
+
\sum_{i=1}^{N-1}\|p_{\pi(i)}-p_{\pi(i-1)}\|+\|p_{\pi(N-1)}-O\|\]

\[
\mathop{\text{minimize}}\limits_{\pi\in S_{MN},p_{k,i}\in\partial\Omega_{k,i}}
\|p_{\pi(1)}-O\|
+
\sum_{h=2}^{MN}\|p_{\pi(h)}-p_{\pi(h-1)}\|+\|p_{\pi(MN)}-O\|
\]

Add one term at last in objective function, which is the distance between beginning and end of path. After discretization, the problem becomes a Hamiltonian path problem with constraints of neighborhoods. 

Thus, with adding one term in objective function in Eq (4), we can solve the optimization. Table 2 lists some numerical results.

\begin{table}[H]
    \centering
    \begin{tabular}{|c|c|c|c|c|}
         \hline
         Angle&  Arc+Tangent&  3 line segments&  4 line segments& 5 line segments\\\hline
         $0\pi/180$ &  2.25565&  2.41421&  2.28454& 2.28456\\\hline
         $10\pi/180$ &  2.34292&  2.5715&  2.38684& 2.35890\\\hline
         $20\pi/180$ &  2.43018&  2.74748&  2.49336& 2.45342\\\hline
         $30\pi/180$ &  2.51745 &  2.9459 &  2.60461 &  2.54986\\\hline
         $40\pi/180$ &  2.60471 &  3.17159 &  2.72113 &  2.64845\\\hline
         $50\pi/180$ &  2.69198 &  3.43084 &  2.84356 &  2.74952\\\hline
         $60\pi/180$ &  2.77925 &  3.73205 &  2.97254 &  2.85360\\\hline
         $70\pi/180$ &  2.86651 &  3.38175 &  3.10886 &  2.89420\\\hline
         $75\pi/180$ &  2.91015 &  3.22939 &  3.08128 &  2.90581\\\hline
         $80\pi/180$ &  2.95378 &  3.08781 &  3.01762 &  2.91467\\\hline
         $85\pi/180$ &  2.99741 &  2.95474] &  2.93872 &  2.90458\\\hline
         $90\pi/180$ &  3.04105 &  2.82843 &  2.82842 &  2.82842\\\hline
         $95\pi/180$ &  3.08468 &  2.96037 &  2.92823 &  2.92680\\\hline
         $100\pi/180$ &  3.12831 &  3.11145 &  2.99243 &  2.99245\\\hline
         $105\pi/180$ &  3.17195 &  3.28536 &  3.07424 &  3.05812\\\hline
         $110\pi/180$ &  3.21558 &  3.48689 &  3.15262 &  3.13283\\\hline
         $115\pi/180$ &  3.25921 &  3.72232 &  3.16857 &  3.15669\\\hline
         $120\pi/180$ &  3.30285 &  4 &  3.15470 &  3.15470\\\hline
         $125\pi/180$ &  3.34648 &  4.33136 &  3.22882 &  3.22885\\\hline
         $130\pi/180$ &  3.39011 &  4.7324 &  3.32581 &  3.32582\\\hline
         $135\pi/180$ &  3.43375 &  5.22625 &  3.43532 &  3.37772\\\hline
         $140\pi/180$ &  3.47738 &  5.84761 &  3.55906 &  3.40907\\\hline
         $145\pi/180$ &  3.52101 &  6.65102 &  3.69914 &  3.48605\\\hline
         $150\pi/180$ &  3.56465 &  7.72741 &  3.85810 &  3.57572\\\hline
         $160\pi/180$ &  3.65191 &  11.5175 &  4.24622 &  3.78904\\\hline
         $180\pi/180$ &  3.82645 &  - &  5.46236 &  4.40378\\\hline
    \end{tabular}
    \caption{Numerical results of searching for two lines with angle $\beta$ with closed path}
    \label{tab:placeholder}
\end{table}

We plot the results in Figure 4, including arc and tangent line and 3-7 line segments.

\begin{figure}[H]
    \centering
    \includegraphics[width=1\linewidth]{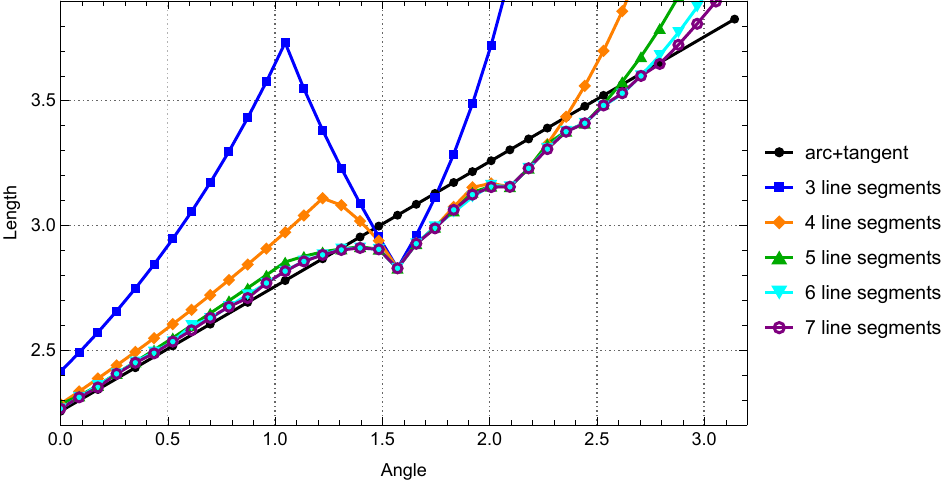}
    \caption{Length of optimal closed path in various angles $\beta$ of searching two lines}
\end{figure}

Similarly, for arc and tangent line, the continuous-form equations can be obtained as
\begin{equation}
\begin{aligned}
\text{minimize} &\sqrt{x(0)^2+y(0)^2}+\int_{0}^{\beta +\pi}  \sqrt{x'(t)^2+y'(t)^2}dt+\sqrt{x(\beta +\pi)^2+y(\beta +\pi)^2}\\
\text{subject to:} & x(t) \cos t + y(t) \sin t - 1 = 0,\forall t\in \left [ 0, \beta +\pi\right ]
\end{aligned}
\end{equation}

Euler-Lagrange equation can be used to find the minimum of this functional. It serves as an upper bound for the solution.

Then for multi-segment polyline, we can have the conversion of coverage solution equation below:

Let $R_1 = \frac{1}{2\cos\gamma_1}$, $R_2 = \frac{1}{2\cos\gamma_2}$, polyline length $\left | OP_1 \right | +\left | P_1P_2 \right |+\left | P_2O \right | $ can be written as 
\begin{equation}
L_{3lines closed} = R_1 + \sqrt{R_1^2 + R_2^2 - 2R_1R_2 \cos(\phi_2 - \phi_1)}+R_2 
\end{equation}

Comparing with Eq(8), we add one term in objective function, which is distance between beginning and end point of path to make length calculation for closed path. 

Then we can find three-line closed optimal path by
\begin{equation}
\begin{array}{l}
\mathop{\text{minimize}}\limits_{\gamma_i, \phi_i } L_{3lines closed} \\
\text{subject to:}0\le \theta_i  \le \pi/2, 0 \le \phi_i \le 2 \pi,\\
([\phi_1 - \gamma_1, \phi_1 + \gamma_1] \\
\cup [\phi_1 - \gamma_1 + \beta, \phi_1 + \gamma_1 + \beta] \\
\cup [\phi_2 - \gamma_2, \phi_2 + \gamma_2] \\
\cup [\phi_2 - \gamma_2 + \beta, \phi_2 + \gamma_2 + \beta]) \bmod 2\pi = [0, 2\pi]
\end{array}
\end{equation}

More line segments are similar to Eq(13-14). More examples of detailed results are shown in Figure 5. Shortest path should be the minimum value for each angle.

It is noted that similar approaches allow for the construction of more line segments, and finer numerical calculations and segmentation can yield more accurate results. 

\begin{figure}[H]
  \centering
  \includegraphics[width=4cm]{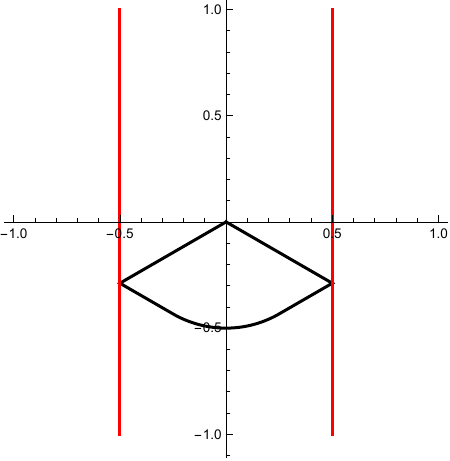}
  \includegraphics[width=4cm]{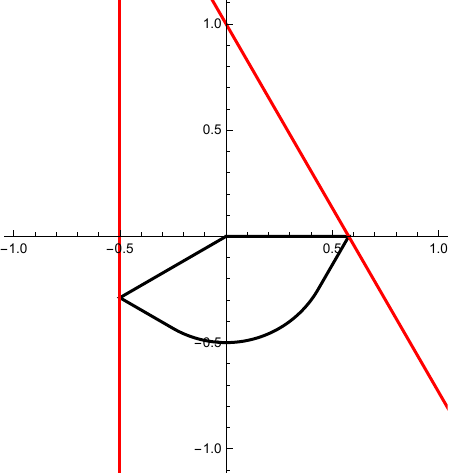}
  \includegraphics[width=4cm]{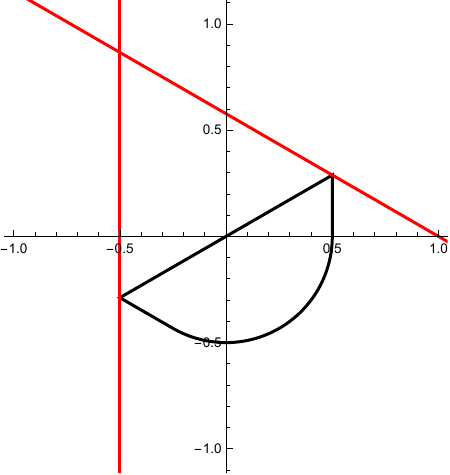}
  \includegraphics[width=4cm]{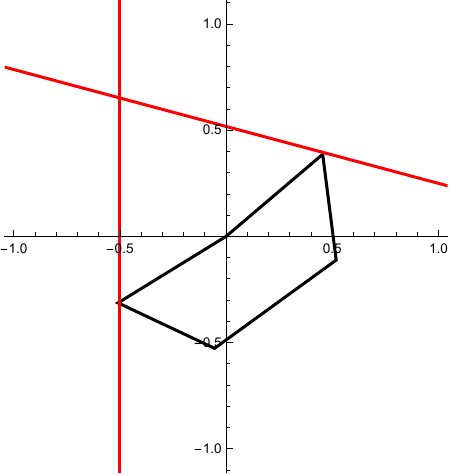}
  \includegraphics[width=4cm]{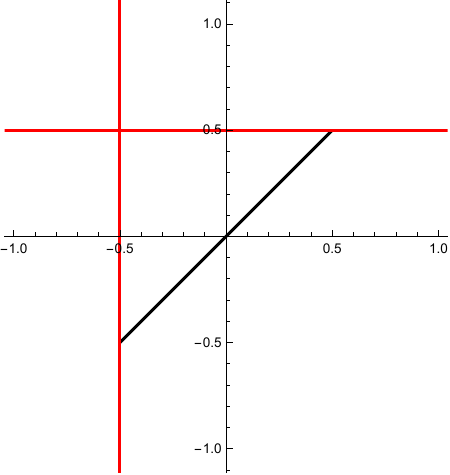}
  \includegraphics[width=4cm]{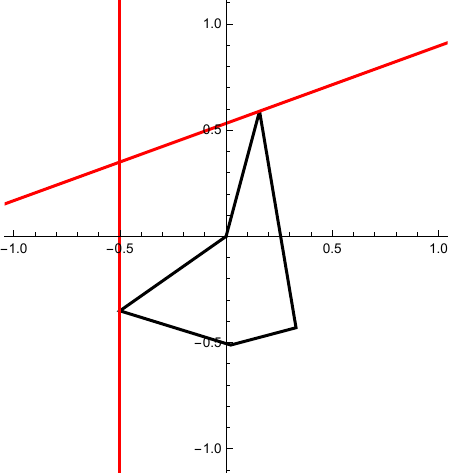}
  \includegraphics[width=4cm]{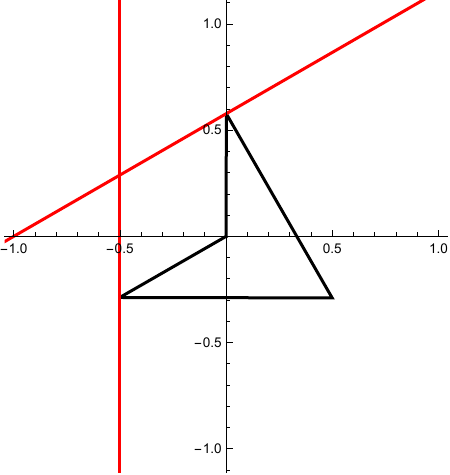}
  \includegraphics[width=4cm]{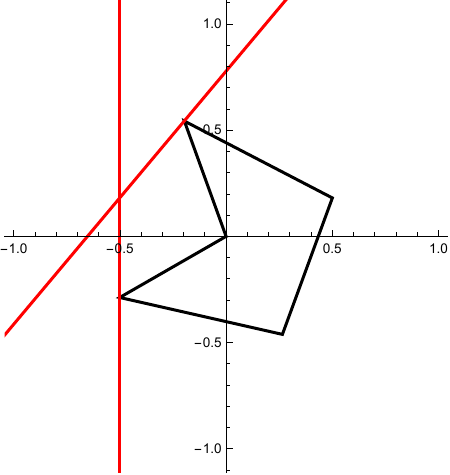}
  \includegraphics[width=4cm]{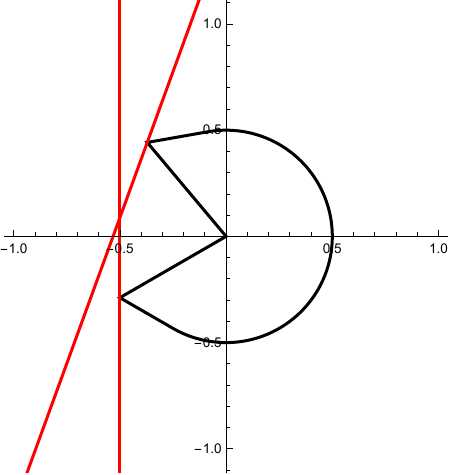}
  \includegraphics[width=4cm]{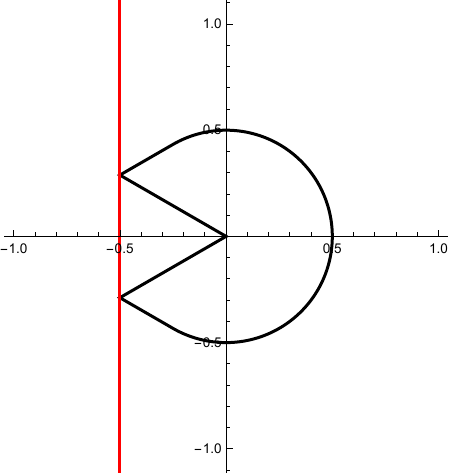}
  \caption{More examples of detailed results of searching for two lines with closed path from the angle bisector in the middle at certain angle $0\pi/180$, $30\pi/180$, $60\pi/180$, $75\pi/180$, $90\pi/180$, $110\pi/180$, $120\pi/180$, $140\pi/180$, $160\pi/180$, $180\pi/180$. Among them, $90\pi/180$ is two line segments, $75\pi/180$,$110\pi/180$, $120\pi/180$ are three line segments, $140\pi/180$ is four line segments. (Black curve is escape path, and red curve is forest boundary)}
\end{figure}

The results of searching for two lines with closed path in this section should provide insights into study polygons (triangles, quadrilaterals, etc) that fit within or covered by other polygons (triangles, quadrilaterals, etc).

\section{Extension of forest and worm problems to three dimensions}
In our previous paper \cite{Deng2024}, we briefly introduced Bellman’s Lost-in-a-Forest Problem and Moser's worm problem in three dimensions. The Bellman’s Lost-in-a-forest Problem in 3D was mentioned in some previous papers \cite{Croft2012} \cite{Zalgaller2003} \cite{Chan2003}  \cite{Croft1969} \cite{Zalgaller1994} \cite{Treeby2026} \cite{Ghomi2021} , and widely known to be more difficult \cite {Brass2005}  \cite{Ghomi2017} \cite{ASTAD1999}.

Following a literature review, we found that previous studies have not specifically described or compared the various dimensions of Bellman’s Lost-in-a-Forest Problem and their relationship to the convex hull \cite{Wetzel2003}. Table 3 presents a comparison of four variations and provides accompanying descriptions.

\begin{table}
    \centering
    \begin{tabular}{|m{3.5cm}|m{2cm}|m{2.5cm}|m{3cm}|m{3cm}|}
    \hline
        \textbf{1D Description and Results} & \textbf{2D Forest Description} & \textbf{2D Convex Hull Description} & \textbf{2D Results Visualization} & \textbf{3D Results Visualization} \\ \hline
        Starting from the origin on a number line, search all points ranging from -n to n. The shortest path is to travel from 0 to n, and then return to -n & Search for all lines with unit distance from origin, path not close & The shortest non-closed curve originating from the origin whose convex hull contains the unit circle & \includegraphics[width=0.2\textwidth]{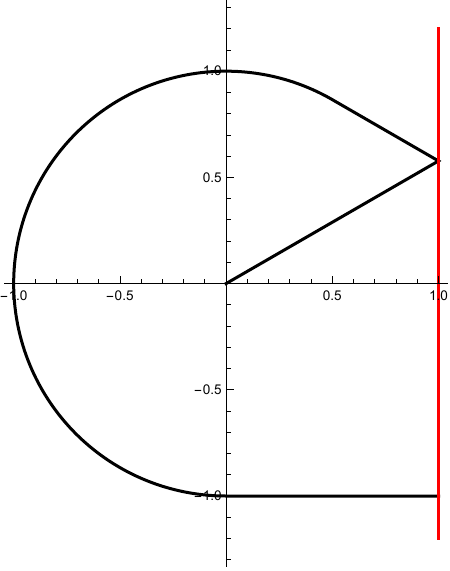} & \includegraphics[width=0.2\textwidth]{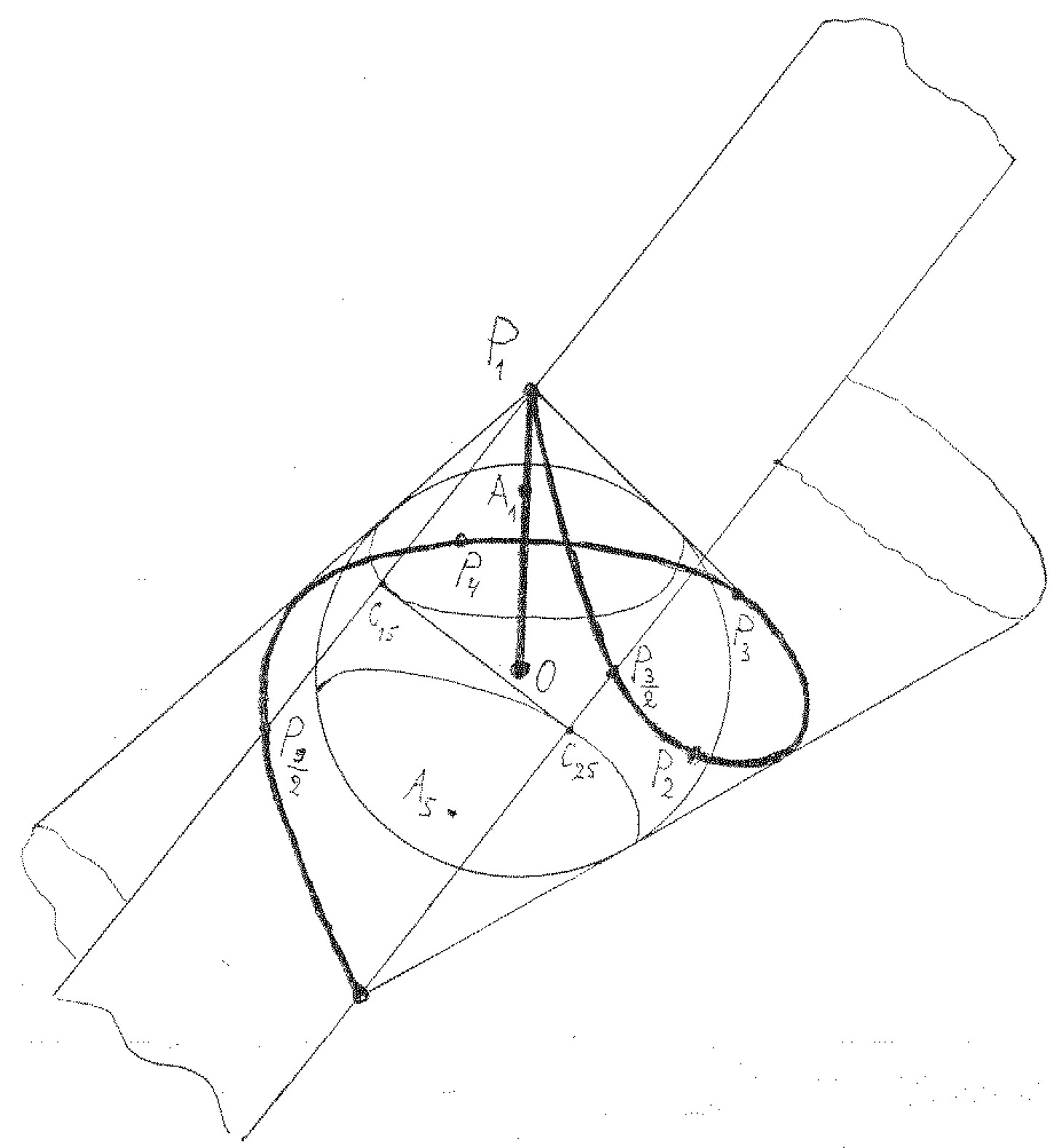} Path length $\approx 11.605$ \cite{Zalgaller2003} \\ \hline
        Search for all points on the number line from -n to n on a number line. The shortest path is from -n to n & Search for all lines with unit distance from any point, path not close & The shortest open curve whose convex hull contains the unit circle & \includegraphics[width=0.2\textwidth]{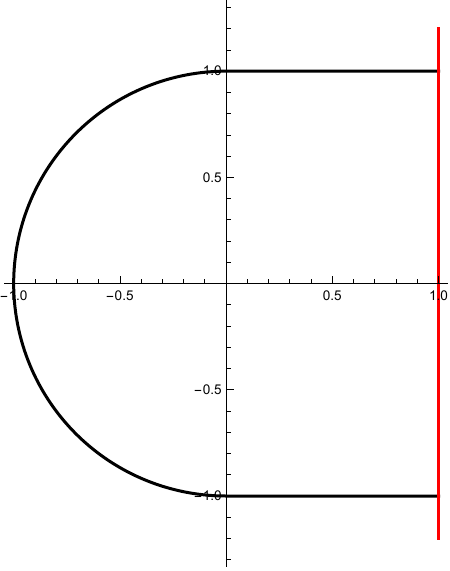} & Not found in the literature \\ \hline
        Starting from origin on a number line, search all points ranging from -n to n, and return to origin. The shortest path involves traveling from 0 to n, then back to -n, and finally returning to the origin & Search for all lines with unit distance from origin and return to origin, path close & The shortest closed curve starting from and returning to the origin, whose convex hull contains the unit circle & \includegraphics[width=0.2\textwidth]{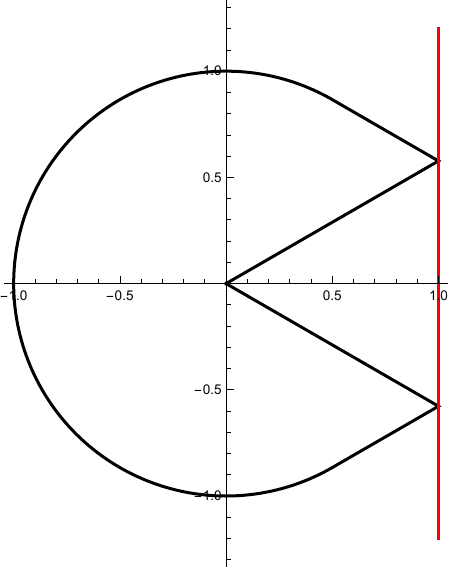} & Not found in the literature\\ \hline
        Search for all points on number line from -n to n and return to starting point. The shortest path is from n to -n to n & Search for all lines with unit distance with path close & The shortest closed curve whose convex hull contains the unit circle & \includegraphics[width=0.2\textwidth]{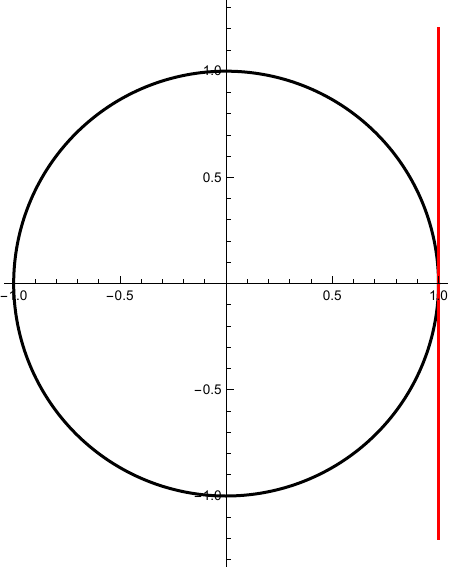} & \includegraphics[width=0.2\textwidth]{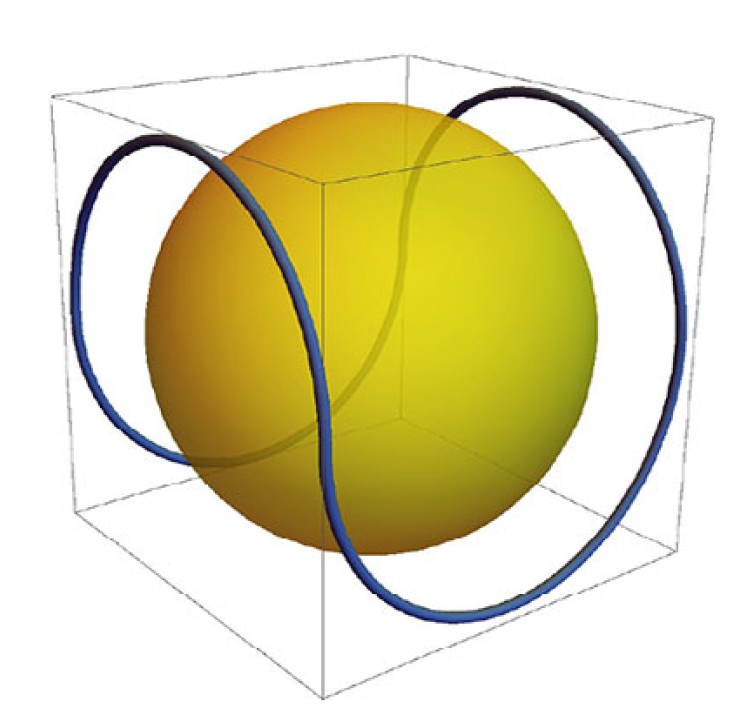} Path length=$4\pi$ \cite{Ghomi2021} \cite{Zalgaller1996} \\ \hline
    \end{tabular}
    \caption{Comparison accompanying descriptions and results of 4 variations of Bellman’s Lost-in-a-forest Problem and convex hull in 1D, 2D, and 3D}
    \label{tab:placeholder}
\end{table}

The descriptions for 3D forests and convex hulls are similar, just replacing the unit circle with unit ball, and the tangent line with tangent plane. Therefore, they are not reiterated in the table.

Based on Weak Form I \cite{Deng2024}, we can obtain the discretized solution equations for variations of Bellman’s Lost-in-a-forest Problem and Convex Hull for 1D, 2D, and 3D listed below:

\begin{itemize}
    \item Variation 1, path starting from origin and not closed

Discretized solution equation in 1D:
\begin{equation}
\begin{aligned}
& \text{minimize} \left\| x_{a_0} - 0\right\|  + \sum_{i=1}^{2n} \left\|x_{a_i} - x_{a_{i-1}}\right\| \\
& \text{subject to:} \{x_{a_0}, x_{a_1}, \dots, x_{a_{2n}}\} \in S_n
\end{aligned}
\end{equation}
where $S_n$  is the symmetric group of all permutations of $\{-n,-n+1, ..., n-1, n\}$

Discretized solution equation in 2D:
\begin{equation}
\begin{aligned}
& \text{minimize} \left\| p_{a_0}-O\right\| +\sum_{i=1}^{N-1}\left\| p_{a_i}-p_{a_{i-1}}\right\| \\
& \text{subject to:} p_{a_i}=(x_{a_i},y_{a_i}), O=(0,0)\\
&x_{a_i} \cos \frac{2\pi i}{N} + y_{a_i} \sin \frac{2\pi i}{N} - 1 = 0, \forall i \in \{0, 1, 2, \dots, N-1\} \\
& \{a_0, a_1, \dots, a_{N-1}\} \in S_N
\end{aligned}
\end{equation}

Discretized solution equation in 3D:
\begin{equation}
\begin{aligned}
& \text{minimize} \left\| p_{a_1}-O\right\| +\sum_{i=2}^{MN}\left\| p_{a_i}-p_{a_{i-1}}\right\| \\
& \text{subject to:} p_{a_i}=(x_{a_i},y_{a_i},z_{a_i}), O=(0,0,0)\\
&x_{a_h} \sin \frac{\pi i}{N} \cos \frac{2\pi k}{M} + y_{a_h} \sin \frac{\pi i}{N} \sin \frac{2\pi k}{M} + z_{a_h} \cos \frac{\pi i}{N} - 1 = 0 \\
& \forall i \in \{1, 2, \dots, N\}, \forall k \in \{1, 2, \dots, M\}, \exists h \in \{1, 2, \dots, MN\}
\end{aligned}
\end{equation}

Zalgaller \cite{Zalgaller2003} presented a 3D construction. Note that although the previous paper \cite{Zalgaller2003} mentioned the optimized length of constructed path 10.605. Since the radius was excluded in the calculation, total path length should be 11.605, as shown in Table 3. We also reproduced this result as shown in Figure 6 (Accurate optimized length result is 10.605427999164). Sections 4 and 6 in this paper used similar method, a segmented approach to summation and optimization.

\begin{figure}
    \centering
    \includegraphics[width=0.5\linewidth]{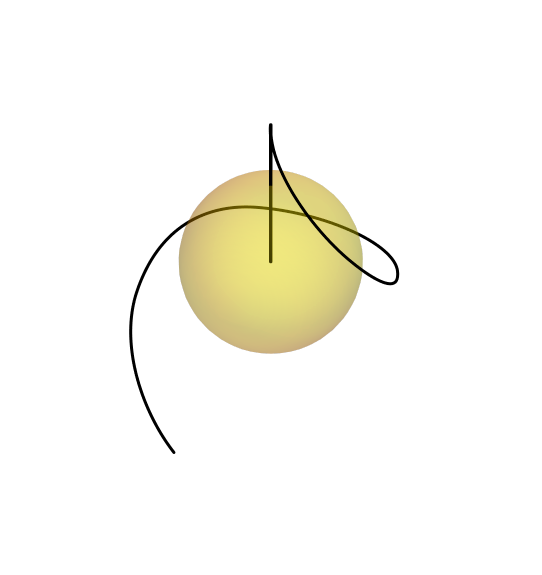}
    \caption{Reproduced non-closed curve originating from the origin whose convex hull contains the unit ball. }
\end{figure}

    \item Variation 2, path not starting from origin and not closed

Discretized solution equation in 1D:
\begin{equation}
\begin{aligned}
& \text{minimize} \sum_{i=1}^{2n} \left\|x_{a_i} - x_{a_{i-1}}\right\| \\
& \text{subject to:} \{x_{a_0}, x_{a_1}, \dots, x_{a_{2n}}\} \in S_n
\end{aligned}
\end{equation}

Discretized solution equation in 2D:
\begin{equation}
\begin{aligned}
& \text{minimize} \sum_{i=1}^{N-1}\left\| p_{a_i}-p_{a_{i-1}}\right\| \\
& \text{subject to:} p_{a_i}=(x_{a_i},y_{a_i})\\& x_{a_i} \cos \frac{2\pi i}{N} + y_{a_i} \sin \frac{2\pi i}{N} - 1 = 0, \forall i \in \{0, 1, 2, \dots, N-1\} \\
& \{a_0, a_1, \dots, a_{N-1}\} \in S_N
\end{aligned}
\end{equation}

Discretized solution equation in 3D:
\begin{equation}
\begin{aligned}
& \text{minimize} \sum_{i=2}^{MN}\left\| p_{a_i}-p_{a_{i-1}}\right\|  \\
& \text{subject to:} p_{a_i}=(x_{a_i},y_{a_i},z_{a_i})\\& x_{a_h} \sin \frac{\pi i}{N} \cos \frac{2\pi k}{M} + y_{a_h} \sin \frac{\pi i}{N} \sin \frac{2\pi k}{M} + z_{a_h} \cos \frac{\pi i}{N} - 1 = 0 \\
& \forall i \in \{1, 2, \dots, N\}, \forall k \in \{1, 2, \dots, M\}, \exists h \in \{1, 2, \dots, MN\}
\end{aligned}
\end{equation}

    \item Variation 3, path starting from origin and closed

Discretized solution equation in 1D:
\begin{equation}
\begin{aligned}
& \text{minimize} \left\| x_{a_0} - 0\right\|  + \sum_{i=1}^{2n} \left\|x_{a_i} - x_{a_{i-1}}\right\| + \left\| x_{a_{2n}} - 0\right\|\\
& \text{subject to:} \{x_{a_0}, x_{a_1}, \dots, x_{a_{2n}}\} \in S_n
\end{aligned}
\end{equation}

Discretized solution equation in 2D:
\begin{equation}
\begin{aligned}
& \text{minimize} \left\| p_{a_0}-O\right\| +\sum_{i=1}^{N-1}\left\| p_{a_i}-p_{a_{i-1}}\right\| +\left\| p_{a_N}-O\right\|\\
& \text{subject to:} p_{a_i}=(x_{a_i},y_{a_i}), O=(0,0)\\
&x_{a_i} \cos \frac{2\pi i}{N} + y_{a_i} \sin \frac{2\pi i}{N} - 1 = 0, \forall i \in \{0, 1, 2, \dots, N-1\} \\
& \{a_0, a_1, \dots, a_{N-1}\} \in S_N
\end{aligned}
\end{equation}

Discretized solution equation in 3D:
\begin{equation}
\begin{aligned}
& \text{minimize} \left\| p_{a_1}-O\right\| +\sum_{i=2}^{MN}\left\| p_{a_i}-p_{a_{i-1}}\right\| +\left\| p_{a_{MN}}-O\right\| \\
& \text{subject to:} p_{a_i}=(x_{a_i},y_{a_i},z_{a_i}), O=(0,0,0)\\&x_{a_h} \sin \frac{\pi i}{N} \cos \frac{2\pi k}{M} + y_{a_h} \sin \frac{\pi i}{N} \sin \frac{2\pi k}{M} + z_{a_h} \cos \frac{\pi i}{N} - 1 = 0 \\
& \forall i \in \{1, 2, \dots, N\}, \forall k \in \{1, 2, \dots, M\}, \exists h \in \{1, 2, \dots, MN\}
\end{aligned}
\end{equation}

    \item Variation 4, path not starting from origin and closed

Discretized solution equation in 1D:
\begin{equation}
\begin{aligned}
& \text{minimize} \left\| x_{a_0} -  x_{a_{2n}}\right\|  + \sum_{i=1}^{2n} \left\|x_{a_i} - x_{a_{i-1}}\right\|\\
& \text{subject to:} \{x_{a_0}, x_{a_1}, \dots, x_{a_{2n}}\} \in S_n
\end{aligned}
\end{equation}

Discretized solution equation in 2D:
\begin{equation}
\begin{aligned}
& \text{minimize} \left\| p_{a_0}-p_{a_{N-1}}\right\| +\sum_{i=1}^{N-1}\left\| p_{a_i}-p_{a_{i-1}}\right\| \\
& \text{subject to:} p_{a_i}=(x_{a_i},y_{a_i})\\
&x_{a_i} \cos \frac{2\pi i}{N} + y_{a_i} \sin \frac{2\pi i}{N} - 1 = 0, \forall i \in \{0, 1, 2, \dots, N-1\} \\
& \{a_0, a_1, \dots, a_{N-1}\} \in S_N
\end{aligned}
\end{equation}

Discretized solution equation in 3D:
\begin{equation}
\begin{aligned}
& \text{minimize} \left\| p_{a_1}- p_{a_{MN}}\right\| +\sum_{i=2}^{MN}\left\| p_{a_i}-p_{a_{i-1}}\right\| \\
& \text{subject to:} p_{a_i}=(x_{a_i},y_{a_i},z_{a_i})\\& x_{a_h} \sin \frac{\pi i}{N} \cos \frac{2\pi k}{M} + y_{a_h} \sin \frac{\pi i}{N} \sin \frac{2\pi k}{M} + z_{a_h} \cos \frac{\pi i}{N} - 1 = 0 \\
& \forall i \in \{1, 2, \dots, N\}, \forall k \in \{1, 2, \dots, M\}, \exists h \in \{1, 2, \dots, MN\}
\end{aligned}
\end{equation}

Zalgaller presented a 3D construction with a path length of $4\pi$ in the paper\cite{Zalgaller1996}, and Ghomi \cite{Ghomi2021} proved, as shown in Table 3.
    
\end{itemize}

Similar methods can be used in even higher dimensions. We just need to adapt the definition and calculation of distance, as well as the constraints, to a high-dimensional setting.

\section{Appendix-Formalized proofs in Lean}

The appendix provides formalized proofs of Theorems 1-7 in Lean 4 code.

The following is the Lean 4 code for Theorem 1 (Equivalence of formulation):
\begin{lstlisting}[style=LeanCode]
import Mathlib

open Set

namespace EscapePath

/-! ## Pure set-theoretic core -/

def trace {ι X : Type*} (f : ι → X) : Set X :=
  Set.range f

def actualPath {ι X Y : Type*} (e : X ≃ Y) (f : ι → X) : ι → Y :=
  fun t => e (f t)

def transformedBoundary {X Y : Type*} (e : X ≃ Y) (B : Set Y) : Set X :=
  e.symm '' B

theorem mem_transformedBoundary_iff
    {X Y : Type*} (e : X ≃ Y) (B : Set Y) (x : X) :
    x ∈ transformedBoundary e B ↔ e x ∈ B := by
  constructor
  · rintro ⟨y, hy, hxy⟩
    have h : e x = y := by
      calc
        e x = e (e.symm y) := congrArg e hxy.symm
        _ = y := e.apply_symm_apply y
    rwa [h]
  · intro hx
    exact ⟨e x, hx, e.symm_apply_apply x⟩

theorem trace_intersection_equiv
    {X Y ι : Type*} (e : X ≃ Y) (f : ι → X) (B : Set Y) :
    (trace (actualPath e f) ∩ B).Nonempty ↔
      (trace f ∩ transformedBoundary e B).Nonempty := by
  constructor
  · rintro ⟨_, ⟨t, rfl⟩, htB⟩
    refine ⟨f t, ⟨t, rfl⟩, ?_⟩
    exact (mem_transformedBoundary_iff e B (f t)).2 htB
  · rintro ⟨_, ⟨t, rfl⟩, htB⟩
    refine ⟨e (f t), ⟨t, rfl⟩, ?_⟩
    exact (mem_transformedBoundary_iff e B (f t)).1 htB

/-! ## Rigid motions in a real normed vector space -/

section RigidMotion

variable {E : Type*} [NormedAddCommGroup E] [NormedSpace ℝ E]

def rigidEquiv (s : E) (Q : E ≃ₗᵢ[ℝ] E) : E ≃ E where
  toFun x := s + Q x
  invFun y := Q.symm (y - s)
  left_inv x := by
    change Q.symm ((s + Q x) - s) = x
    rw [show (s + Q x) - s = Q x by abel]
    exact Q.symm_apply_apply x
  right_inv y := by
    change s + Q (Q.symm (y - s)) = y
    rw [Q.apply_symm_apply]
    abel

@[simp] theorem rigidEquiv_apply
    (s : E) (Q : E ≃ₗᵢ[ℝ] E) (x : E) :
    rigidEquiv s Q x = s + Q x := rfl

@[simp] theorem rigidEquiv_symm_apply
    (s : E) (Q : E ≃ₗᵢ[ℝ] E) (y : E) :
    (rigidEquiv s Q).symm y = Q.symm (y - s) := rfl

def physicalPath {ι : Type*}
    (s : E) (Q : E ≃ₗᵢ[ℝ] E) (f : ι → E) : ι → E :=
  actualPath (rigidEquiv s Q) f

def pullbackBoundary
    (s : E) (Q : E ≃ₗᵢ[ℝ] E) (B : Set E) : Set E :=
  transformedBoundary (rigidEquiv s Q) B

theorem rigid_motion_equivalence
    {ι : Type*} (s : E) (Q : E ≃ₗᵢ[ℝ] E)
    (f : ι → E) (B : Set E) :
    (trace (physicalPath s Q f) ∩ B).Nonempty ↔
      (trace f ∩ pullbackBoundary s Q B).Nonempty := by
  simpa [physicalPath, pullbackBoundary] using
    (trace_intersection_equiv (rigidEquiv s Q) f B)

/-! A concrete model of the plane: `ℂ ≅ ℝ²`. -/

noncomputable def planeRotation (α : ℝ) : ℂ ≃ₗᵢ[ℝ] ℂ := by
  letI : Fact (Module.finrank ℝ ℂ = 2) :=
    Complex.finrank_real_complex_fact
  exact Complex.orientation.rotation (α : Real.Angle)

/-! Universal feasibility over starting points and orientations. -/

def physicalFeasible
    {A ι : Type*} (B R : Set E) (Θ : Set A)
    (Q : A → E ≃ₗᵢ[ℝ] E) (f : ι → E) : Prop :=
  ∀ s ∈ R, ∀ α ∈ Θ,
    (trace (physicalPath s (Q α) f) ∩ B).Nonempty

def pullbackFeasible
    {A ι : Type*} (B R : Set E) (Θ : Set A)
    (Q : A → E ≃ₗᵢ[ℝ] E) (f : ι → E) : Prop :=
  ∀ s ∈ R, ∀ α ∈ Θ,
    (trace f ∩ pullbackBoundary s (Q α) B).Nonempty

theorem feasible_equivalence
    {A ι : Type*} (B R : Set E) (Θ : Set A)
    (Q : A → E ≃ₗᵢ[ℝ] E) (f : ι → E) :
    physicalFeasible B R Θ Q f ↔ pullbackFeasible B R Θ Q f := by
  constructor
  · intro h s hs α hα
    exact (rigid_motion_equivalence s (Q α) f B).1 (h s hs α hα)
  · intro h s hs α hα
    exact (rigid_motion_equivalence s (Q α) f B).2 (h s hs α hα)

end RigidMotion

/-! ## Distance-zero criterion once a minimum is attained -/

section Distance

variable {M ι : Type*} [PseudoMetricSpace M]

theorem attained_curve_distance_zero_iff_intersects
    (f : ι → M) (B : Set M) (t₀ : ι)
    (hBclosed : IsClosed B) (hBne : B.Nonempty)
    (hmin : ∀ t, Metric.infDist (f t₀) B ≤ Metric.infDist (f t) B) :
    Metric.infDist (f t₀) B = 0 ↔ (Set.range f ∩ B).Nonempty := by
  constructor
  · intro hzero
    have hmem : f t₀ ∈ B :=
      (hBclosed.mem_iff_infDist_zero hBne).2 hzero
    exact ⟨f t₀, ⟨t₀, rfl⟩, hmem⟩
  · rintro ⟨x, ⟨t, rfl⟩, hmem⟩
    have hle : Metric.infDist (f t₀) B ≤ 0 := by
      calc
        Metric.infDist (f t₀) B ≤ Metric.infDist (f t) B := hmin t
        _ = 0 := Metric.infDist_zero_of_mem hmem
    exact le_antisymm hle Metric.infDist_nonneg

theorem exists_infDist_minimizer
    [TopologicalSpace ι] [CompactSpace ι] [Nonempty ι]
    (f : ι → M) (B : Set M) (hf : Continuous f) :
    ∃ t₀, ∀ t, Metric.infDist (f t₀) B ≤ Metric.infDist (f t) B := by
  have hcont : Continuous (fun t => Metric.infDist (f t) B) :=
    (Metric.continuous_infDist_pt B).comp hf
  obtain ⟨t₀, _, hmin⟩ :=
    isCompact_univ.exists_isMinOn Set.univ_nonempty hcont.continuousOn
  exact ⟨t₀, fun t => hmin (Set.mem_univ t)⟩

end Distance

end EscapePath
\end{lstlisting}

The following is the Lean 4 code for Theorem 2 (Existence of minimizer):
\begin{lstlisting}[style=LeanCode]
import Mathlib

open Set

namespace EscapePath

/-!
# Direct-method existence theorem for escape paths

The file separates the proof into three logically independent pieces:

1. a general compact-sublevel minimization theorem;
2. a normalization theorem schema (for arclength reparametrization);
3. closedness of the boundary-hitting constraints.

For the geometric application, instantiate `X` with a uniform-function space of
curves, `F` with total variation (`eVariationOn`), and `normalize` with the
constant-speed/arclength reparametrization.
-/

section DirectMethod

variable {X Y : Type*} [TopologicalSpace X] [LinearOrder Y]

/-- A lower-semicontinuous objective has a global minimizer on `A` as soon as
one sublevel set based at a feasible reference point is compact. -/
theorem exists_global_minimizer_of_compact_sublevel
    (F : X → Y) (A : Set X) (x₀ : X)
    (hx₀ : x₀ ∈ A)
    (hcompact : IsCompact {x | x ∈ A ∧ F x ≤ F x₀})
    (hlsc : LowerSemicontinuousOn F {x | x ∈ A ∧ F x ≤ F x₀}) :
    ∃ x ∈ A, ∀ y ∈ A, F x ≤ F y := by
  let K : Set X := {x | x ∈ A ∧ F x ≤ F x₀}
  change IsCompact K at hcompact
  change LowerSemicontinuousOn F K at hlsc
  have hx₀K : x₀ ∈ K := ⟨hx₀, le_rfl⟩
  obtain ⟨x, hxK, hx_min⟩ :=
    hlsc.exists_isMinOn ⟨x₀, hx₀K⟩ hcompact
  refine ⟨x, hxK.1, ?_⟩
  intro y hyA
  by_cases hyK : F y ≤ F x₀
  · exact hx_min ⟨hyA, hyK⟩
  · exact hxK.2.trans (le_of_lt (lt_of_not_ge hyK))

/-- Direct method with a trace/length-preserving normalization map.

This is the exact abstract form needed for arclength reparametrization:
every feasible object can be normalized without changing its cost, while the
normalized reference sublevel is compact. -/
theorem exists_global_minimizer_via_normalization
    (F : X → Y)
    (Feasible Normalized : Set X)
    (normalize : X → X)
    (x₀ : X)
    (hx₀ : x₀ ∈ Feasible)
    (normalize_feasible : ∀ ⦃x⦄, x ∈ Feasible → normalize x ∈ Feasible)
    (normalize_normalized : ∀ ⦃x⦄, x ∈ Feasible → normalize x ∈ Normalized)
    (normalize_cost : ∀ ⦃x⦄, x ∈ Feasible → F (normalize x) = F x)
    (hcompact :
      IsCompact
        {x | x ∈ Feasible ∩ Normalized ∧ F x ≤ F (normalize x₀)})
    (hlsc :
      LowerSemicontinuousOn F
        {x | x ∈ Feasible ∩ Normalized ∧ F x ≤ F (normalize x₀)}) :
    ∃ x ∈ Feasible, ∀ y ∈ Feasible, F x ≤ F y := by
  let A : Set X := Feasible ∩ Normalized
  have hx₀A : normalize x₀ ∈ A :=
    ⟨normalize_feasible hx₀, normalize_normalized hx₀⟩
  obtain ⟨x, hxA, hx_min⟩ :=
    exists_global_minimizer_of_compact_sublevel
      F A (normalize x₀) hx₀A hcompact hlsc
  refine ⟨x, hxA.1, ?_⟩
  intro y hy
  calc
    F x ≤ F (normalize y) :=
      hx_min (normalize y) ⟨normalize_feasible hy, normalize_normalized hy⟩
    _ = F y := normalize_cost hy

end DirectMethod

section RigidMotion

variable {E : Type*} [TopologicalSpace E]

/-- Canonical-coordinate image of a boundary under the inverse rigid motion. -/
def transformedBoundary (e : E ≃ₜ E) (B : Set E) : Set E :=
  e.symm '' B

@[simp]
theorem mem_transformedBoundary_iff
    (e : E ≃ₜ E) (B : Set E) (x : E) :
    x ∈ transformedBoundary e B ↔ e x ∈ B := by
  constructor
  · rintro ⟨y, hy, rfl⟩
    simpa using hy
  · intro hx
    exact ⟨e x, hx, by simp⟩

variable {ι : Type*}

/-- The actual path in physical coordinates. -/
def actualPath (e : E ≃ₜ E) (f : ι → E) : ι → E :=
  fun t => e (f t)

/-- Hitting the physical boundary is equivalent to hitting its transformed
copy in canonical coordinates. -/
theorem actualPath_hits_iff
    (e : E ≃ₜ E) (B : Set E) (f : ι → E) :
    (∃ t, actualPath e f t ∈ B) ↔
      ∃ t, f t ∈ transformedBoundary e B := by
  simp only [actualPath, mem_transformedBoundary_iff]

end RigidMotion

section ClosedFeasibility

variable {ι E G : Type*}
variable [TopologicalSpace ι] [CompactSpace ι]
variable [MetricSpace E]

abbrev Curve := BoundedContinuousFunction ι E

/-- A curve hits a set if one of its parameter values lies in that set. -/
def Hits (f : Curve (ι := ι) (E := E)) (S : Set E) : Prop :=
  ∃ t : ι, f t ∈ S

/-- Curves hitting a fixed closed set form a closed subset of the uniform
function space. Compactness of the parameter domain is essential here. -/
theorem isClosed_setOf_hits (S : Set E) (hS : IsClosed S) :
    IsClosed {f : Curve (ι := ι) (E := E) | Hits f S} := by
  let C : Set (ι × Curve (ι := ι) (E := E)) :=
    {p | p.2 p.1 ∈ S}
  have hEval :
      Continuous (fun p : ι × Curve (ι := ι) (E := E) => p.2 p.1) := by
    fun_prop
  have hC : IsClosed C := hS.preimage hEval
  have hproj : IsClosed (Prod.snd '' C) :=
    isClosedMap_snd_of_compactSpace C hC
  have heq :
      Prod.snd '' C =
        {f : Curve (ι := ι) (E := E) | Hits f S} := by
    ext f
    constructor
    · rintro ⟨p, hp, rfl⟩
      exact ⟨p.1, by simpa [C] using hp⟩
    · rintro ⟨t, ht⟩
      refine ⟨(t, f), ?_, rfl⟩
      simpa [C] using ht
  rwa [← heq]

/-- Feasibility for an arbitrary family of transformed boundaries. -/
def Feasible
    (boundary : G → Set E)
    (f : Curve (ι := ι) (E := E)) : Prop :=
  ∀ g, Hits f (boundary g)

/-- An arbitrary intersection of closed boundary-hitting constraints is closed.
No compactness assumption on the index type `G` is needed. -/
theorem isClosed_setOf_feasible
    (boundary : G → Set E)
    (hboundary : ∀ g, IsClosed (boundary g)) :
    IsClosed
      {f : Curve (ι := ι) (E := E) | Feasible boundary f} := by
  have heq :
      {f : Curve (ι := ι) (E := E) | Feasible boundary f} =
        ⋂ g, {f : Curve (ι := ι) (E := E) | Hits f (boundary g)} := by
    ext f
    simp [Feasible]
  rw [heq]
  exact isClosed_iInter fun g => isClosed_setOf_hits (boundary g) (hboundary g)

/-- A compact boundary remains compact, hence closed, after a homeomorphism. -/
theorem transformedBoundary_isClosed
    (B : Set E) (hB : IsCompact B) (e : E ≃ₜ E) :
    IsClosed (transformedBoundary e B) := by
  exact (hB.image e.symm.continuous).isClosed

/-- Feasibility for a family of rigid motions is closed in the uniform topology. -/
theorem rigid_feasible_isClosed
    (B : Set E) (hB : IsCompact B) (rigid : G → E ≃ₜ E) :
    IsClosed
      {f : Curve (ι := ι) (E := E) |
        ∀ g, Hits f (transformedBoundary (rigid g) B)} := by
  exact isClosed_setOf_feasible
    (fun g => transformedBoundary (rigid g) B)
    (fun g => transformedBoundary_isClosed B hB (rigid g))

end ClosedFeasibility

section VariationLength

variable {ι E : Type*}
variable [LinearOrder ι] [TopologicalSpace ι]
variable [PseudoMetricSpace E]

/-- Extended-real total variation, the metric definition of curve length. -/
noncomputable def curveLength
    (f : BoundedContinuousFunction ι E) : ENNReal :=
  eVariationOn f Set.univ

end VariationLength

end EscapePath
\end{lstlisting}

The following is the Lean 4 code for Lemma 3, Lemma 4, and Theorem 5 (Discretization yields TSPN):
\begin{lstlisting}[style=LeanCode]
import Mathlib

open Set
open scoped ENNReal

/-!
# Exact finite reduction to a rooted open TSPN

This file formalizes the logical core of the reduction.

The analytic input from absolutely continuous Euclidean curves is isolated in
`RectifiablePath.chain_le_length`: sampling a curve at nondecreasing times gives
an ordered polygonal chain whose length is no greater than the curve length.

The Euclidean construction of a piecewise-linear interpolant is isolated in
`PolylineRealizer`: every finite TSPN tour can be realized by a feasible path
with no larger length.  Once these two standard analytic facts are supplied,
the equality of the two infima is completely formal.
-/

namespace DiscretizedTSPN

universe u v w


/-! ## 1. Rigid-motion/intersection core -/

def Intersects {I X : Type*} (f : I → X) (S : Set X) : Prop :=
  ∃ t, f t ∈ S

def actualPath {I X : Type*} (e : X ≃ X) (f : I → X) : I → X :=
  fun t => e (f t)

def transformedBoundary {X : Type*} (e : X ≃ X) (B : Set X) : Set X :=
  e.symm '' B

theorem mem_transformedBoundary_iff
    {X : Type*} (e : X ≃ X) (B : Set X) (x : X) :
    x ∈ transformedBoundary e B ↔ e x ∈ B := by
  constructor
  · rintro ⟨y, hy, rfl⟩
    simpa using hy
  · intro hx
    exact ⟨e x, hx, by simp⟩

theorem actualPath_intersects_iff
    {I X : Type*} (e : X ≃ X) (f : I → X) (B : Set X) :
    Intersects (actualPath e f) B ↔ Intersects f (transformedBoundary e B) := by
  constructor
  · rintro ⟨t, ht⟩
    exact ⟨t, (mem_transformedBoundary_iff e B (f t)).2 ht⟩
  · rintro ⟨t, ht⟩
    exact ⟨t, (mem_transformedBoundary_iff e B (f t)).1 ht⟩

/-! ## 2. Distance zero versus membership

For a nonempty closed set, Mathlib's point-to-set `infDist` vanishes exactly on
that set. Compact transformed boundaries satisfy these hypotheses.
-/

theorem mem_iff_infDist_zero
    {X : Type*} [PseudoMetricSpace X]
    {S : Set X} (hSclosed : IsClosed S) (hSne : S.Nonempty) (x : X) :
    x ∈ S ↔ Metric.infDist x S = 0 := by
  exact hSclosed.mem_iff_infDist_zero hSne

/-! ## 3. Abstract optimization values and monotonicity -/

def FeasibleOn
    {Γ Path : Type*} (A : Set Γ) (hit : Path → Γ → Prop) (p : Path) : Prop :=
  ∀ g ∈ A, hit p g

/-- The infimum cost over paths feasible on `A`; the value is `∞` if there is
no feasible path. -/
noncomputable def OptValue
    {Γ Path : Type*} (A : Set Γ) (hit : Path → Γ → Prop)
    (cost : Path → ℝ≥0∞) : ℝ≥0∞ :=
  ⨅ p : {p : Path // FeasibleOn A hit p}, cost p.1

theorem optValue_antitone
    {Γ Path : Type*} {A C : Set Γ} (hAC : A ⊆ C)
    (hit : Path → Γ → Prop) (cost : Path → ℝ≥0∞) :
    OptValue A hit cost ≤ OptValue C hit cost := by
  apply le_iInf
  intro pC
  exact iInf_le_of_le
    (⟨pC.1, fun g hg => pC.2 g (hAC hg)⟩ :
      {p : Path // FeasibleOn A hit p})
    le_rfl

theorem nested_optValue_monotone
    {Γ Path : Type*} (G : ℕ → Set Γ)
    (h_nested : ∀ m, G m ⊆ G (m + 1))
    (hit : Path → Γ → Prop) (cost : Path → ℝ≥0∞) (m : ℕ) :
    OptValue (G m) hit cost ≤ OptValue (G (m + 1)) hit cost := by
  exact optValue_antitone (h_nested m) hit cost

theorem discretized_value_le_full_value
    {Γ Path : Type*} (A G : Set Γ) (hAG : A ⊆ G)
    (hit : Path → Γ → Prop) (cost : Path → ℝ≥0∞) :
    OptValue A hit cost ≤ OptValue G hit cost := by
  exact optValue_antitone hAG hit cost

/-! ## 4. Rooted open TSPN

The route begins at `origin`, visits one selected point from every neighborhood,
and has no forced return edge.
-/

variable {X : Type u} [PseudoMetricSpace X]

/-- Length of the ordered chain `x₀ → q₁ → ... → qₙ`. -/
def chainCost (x₀ : X) : List X → ℝ≥0∞
  | [] => 0
  | q :: qs => ENNReal.ofReal (dist x₀ q) + chainCost q qs

@[simp] theorem chainCost_nil (x₀ : X) : chainCost x₀ [] = 0 := rfl

@[simp] theorem chainCost_cons (x₀ q : X) (qs : List X) :
    chainCost x₀ (q :: qs) =
      ENNReal.ofReal (dist x₀ q) + chainCost q qs := rfl

/-- A rectifiable-path interface containing exactly the chain inequality used by
polygonal shortening.  Absolutely continuous Euclidean curves with integral
length instantiate this structure. -/
structure RectifiablePath (origin : X) where
  toFun : ℝ → X
  start : toFun 0 = origin
  length : ℝ≥0∞
  chain_le_length :
    ∀ {H : ℕ} (time : Fin H → ℝ),
      (∀ i, time i ∈ Set.Icc (0 : ℝ) 1) →
      ∀ order : List (Fin H),
        order.Nodup →
        (∀ i, i ∈ order) →
        order.Pairwise (fun i j => time i ≤ time j) →
        chainCost origin (order.map (fun i => toFun (time i))) ≤ length

instance (origin : X) : CoeFun (RectifiablePath origin) (fun _ => ℝ → X) :=
  ⟨RectifiablePath.toFun⟩

/-- A rooted open tour with one chosen representative in each neighborhood. -/
structure Tour {H : ℕ} (S : Fin H → Set X) where
  order : List (Fin H)
  nodup : order.Nodup
  complete : ∀ i, i ∈ order
  point : Fin H → X
  point_mem : ∀ i, point i ∈ S i

def tourCost
    {H : ℕ} (origin : X) {S : Fin H → Set X} (τ : Tour S) : ℝ≥0∞ :=
  chainCost origin (τ.order.map τ.point)

/-- A path is feasible when it hits every neighborhood at some parameter time
in `[0,1]`. -/
def FiniteFeasible
    {H : ℕ} (S : Fin H → Set X) {origin : X}
    (p : RectifiablePath origin) : Prop :=
  ∀ i, ∃ t ∈ Set.Icc (0 : ℝ) 1, p t ∈ S i

/-- Sorting one selected hit time per neighborhood produces a valid tour whose
cost is bounded by the path length.  Lexicographic tie-breaking by the finite
index handles simultaneous hits rigorously. -/
theorem exists_tour_cost_le_of_feasible
    {H : ℕ} {origin : X} {S : Fin H → Set X}
    (p : RectifiablePath origin) (hp : FiniteFeasible S p) :
    ∃ τ : Tour S, tourCost origin τ ≤ p.length := by
  classical
  let hitTime : Fin H → ℝ := fun i => Classical.choose (hp i)
  have hitTime_mem : ∀ i, hitTime i ∈ Set.Icc (0 : ℝ) 1 := by
    intro i
    exact (Classical.choose_spec (hp i)).1
  have hitPoint_mem : ∀ i, p (hitTime i) ∈ S i := by
    intro i
    exact (Classical.choose_spec (hp i)).2
  let key : Fin H → ℝ ×ₗ Fin H :=
    fun i => toLex (hitTime i, i)
  have key_injective : Function.Injective key := by
    intro i j hij
    have hsecond := congrArg (fun z : ℝ ×ₗ Fin H => (ofLex z).2) hij
    simpa [key] using hsecond
  /- We sort with the pullback of lexicographic order along `key`, rather than
  replacing the existing `LinearOrder (Fin H)`. This avoids ambiguity between
  the native order on `Fin H` and the temporary lifted order. -/
  let rel : Fin H → Fin H → Prop := fun i j => key i ≤ key j
  letI : DecidableRel rel := by
    intro i j
    exact inferInstanceAs (Decidable (key i ≤ key j))
  letI : IsTrans (Fin H) rel :=
    ⟨by
      intro i j k hij hjk
      have hij' : key i ≤ key j := by
        simpa [rel] using hij
      have hjk' : key j ≤ key k := by
        simpa [rel] using hjk
      simpa [rel] using hij'.trans hjk'⟩
  letI : Std.Antisymm rel :=
    ⟨by
      intro i j hij hji
      have hij' : key i ≤ key j := by
        simpa [rel] using hij
      have hji' : key j ≤ key i := by
        simpa [rel] using hji
      exact key_injective (le_antisymm hij' hji')⟩
  letI : Std.Total rel :=
    ⟨by
      intro i j
      rcases le_total (key i) (key j) with hij | hji
      · exact Or.inl (by simpa [rel] using hij)
      · exact Or.inr (by simpa [rel] using hji)⟩
  let order : List (Fin H) := Finset.univ.sort rel
  have order_nodup : order.Nodup := by
    dsimp [order]
    exact Finset.sort_nodup _ _
  have order_complete : ∀ i, i ∈ order := by
    intro i
    simp [order]
  have order_pairwise_key : order.Pairwise rel := by
    dsimp [order]
    exact Finset.pairwise_sort _ _
  have order_pairwise_time :
      order.Pairwise (fun i j : Fin H => hitTime i ≤ hitTime j) := by
    refine order_pairwise_key.imp ?_
    intro i j hij
    have hkey : key i ≤ key j := by
      simpa [rel] using hij
    simpa [key] using Prod.Lex.monotone_fst (key i) (key j) hkey
  let τ : Tour S :=
    { order := order
      nodup := order_nodup
      complete := order_complete
      point := fun i => p (hitTime i)
      point_mem := hitPoint_mem }
  refine ⟨τ, ?_⟩
  simpa [tourCost, τ] using
    p.chain_le_length hitTime hitTime_mem order order_nodup
      order_complete order_pairwise_time

/-- The discretized path optimum. -/
noncomputable def pathValue
    {H : ℕ} (origin : X) (S : Fin H → Set X) : ℝ≥0∞ :=
  ⨅ p : {p : RectifiablePath origin // FiniteFeasible S p}, p.1.length

/-- The rooted open TSPN optimum. -/
noncomputable def tspnValue
    {H : ℕ} (origin : X) (S : Fin H → Set X) : ℝ≥0∞ :=
  ⨅ τ : Tour S, tourCost origin τ

/-- The reverse analytic input: every selected ordered neighborhood tour is
realized by a feasible path no longer than its polygonal-chain objective.
In Euclidean space, the ordinary piecewise-linear interpolant satisfies this. -/
structure PolylineRealizer
    {H : ℕ} (origin : X) (S : Fin H → Set X) where
  realize : Tour S → RectifiablePath origin
  feasible : ∀ τ, FiniteFeasible S (realize τ)
  length_le : ∀ τ, (realize τ).length ≤ tourCost origin τ

/-- Every feasible path admits a no-longer TSPN tour. -/
theorem tspnValue_le_pathValue
    {H : ℕ} (origin : X) (S : Fin H → Set X) :
    tspnValue origin S ≤ pathValue origin S := by
  apply le_iInf
  intro p
  rcases exists_tour_cost_le_of_feasible p.1 p.2 with ⟨τ, hτ⟩
  calc
    tspnValue origin S ≤ tourCost origin τ :=
      iInf_le (fun σ : Tour S => tourCost origin σ) τ
    _ ≤ p.1.length := hτ

/-- A polyline realizer gives the reverse inequality. -/
theorem pathValue_le_tspnValue
    {H : ℕ} (origin : X) (S : Fin H → Set X)
    (R : PolylineRealizer origin S) :
    pathValue origin S ≤ tspnValue origin S := by
  apply le_iInf
  intro τ
  calc
    pathValue origin S ≤ (R.realize τ).length :=
      iInf_le
        (fun p : {p : RectifiablePath origin // FiniteFeasible S p} => p.1.length)
        (⟨R.realize τ, R.feasible τ⟩ :
          {p : RectifiablePath origin // FiniteFeasible S p})
    _ ≤ tourCost origin τ := R.length_le τ

/-- Exact equality of the discretized problem and rooted open TSPN. -/
theorem discretized_problem_eq_tspn
    {H : ℕ} (origin : X) (S : Fin H → Set X)
    (R : PolylineRealizer origin S) :
    pathValue origin S = tspnValue origin S := by
  apply le_antisymm
  · exact pathValue_le_tspnValue origin S R
  · exact tspnValue_le_pathValue origin S

end DiscretizedTSPN
\end{lstlisting}

The following is the Lean 4 code for Theorem 6, 7, 8, 9 (Certificate and convergence):
\lstinputlisting[style=LeanCode]{EscapePathCertifiedTSPNConvergence_v3.lean}

\textbf{AI usage disclosure:} The formalized proofs in Lean 4 code was assisted by GPT-5.6 sol.

~\\
College of Engineering and Computer Science, University of Central Florida, Orlando, FL, USA

Email: \underline{zhipeng.deng@ucf.edu}

\end{document}